\pdfoutput=1
\documentclass[11pt,leqno]{amsart}\usepackage{epsfig,amsmath,
amssymb,graphicx}
\usepackage{pinlabel}
\usepackage{amscd}
\usepackage{multirow}
\usepackage[linktocpage]{hyperref}
\usepackage[left=1.25in,top=1.2in,right=1.25in,bottom=1.2in,head=.2in]{geometry}
\DeclareGraphicsExtensions{.pdf}
\usepackage[arrow,matrix,curve,cmtip,ps]{xy}
\usepackage{pb-diagram} 
\usepackage{pb-xy}
\usepackage{tikz}

\newcommand{\lk}{\operatorname{lk}}

\newcommand{\Z}{\mathbb Z}

\newcommand{\R}{\mathbb R}

\newcommand{\nhd}{\operatorname{nhd}}

\newcommand{\Wh}{\operatorname{Wh}}

\newcommand{\id}{\operatorname{id}}

\newcommand{\inte}{\operatorname{int}}

\def\conn{\mathbin{\#}}

\newcommand{\sss}{{S^2\hskip-2pt\times\hskip-2pt S^2}}

\newcommand{\ytilde}{\widetilde Y}
\newcommand{\vtilde}{\widetilde V}
\newcommand{\wtilde}{\widetilde W}

\newcommand{\gtilde}{\tilde\gamma}
\newcommand{\St}{\widetilde S}

\newcommand{\ba}{{\bf a}}
\newcommand{\bb}{{\bf b}}

\renewcommand{\phi}{\varphi}

\newcommand{\spinc}{\ifmmode{\operatorname{Spin}^c}\else{$\operatorname{spin}^c$\ }\fi}

\newcommand{\tb}{\operatorname{tb}}
\newcommand{\rot}{\operatorname{rot}}

\newtheorem{theorem}{Theorem}[section]

\newtheorem*{T:introsphere}{Theorem~\ref{T:nodehnforsphere}}
\newtheorem*{T:smooth-ball}{Theorem~\ref{T:smooth-ball}}
\newtheorem*{T:nodehnfortori}{Theorem~\ref{T:nodehnfortori}}
\newtheorem*{T:torus-smooth}{Theorem~\ref{T:torus-smooth}}

\newtheorem{lemma}[theorem]{Lemma}
\newtheorem{proposition}[theorem]{Proposition}
\newtheorem{corollary}[theorem]{Corollary}

\theoremstyle{definition}

\def\acknowledgementname{Acknowledgements.}
\newenvironment{acknowledgement}
   {
    \vspace*{\baselineskip}
    \noindent{\small\textbf{\acknowledgementname}}
    \unskip\noindent}{}    
\title{Four-dimensional analogues of Dehn's lemma}
\author[Arunima Ray]{Arunima Ray}
\address{Department of Mathematics, MS 050\newline\indent Brandeis
University \newline\indent Waltham, MA 02454}
\email{aruray@brandeis.edu}
\author[Daniel Ruberman]{Daniel Ruberman${}^1$}
\address{Department of Mathematics, MS 050\newline\indent Brandeis
University \newline\indent Waltham, MA 02454}
\email{ruberman@brandeis.edu}

\date{\today}
\thanks{$^1$Partially supported by NSF Grant 1506328}

\begin{document}
\maketitle
\begin{abstract}
We investigate certain $4$-dimensional analogues of the classical $3$-dimensional Dehn's lemma, giving examples where such analogues do or do not hold, in the smooth and topological categories. In particular, we show that an essential $2$-sphere $S$ in the boundary of a simply connected $4$-manifold $W$ such that $S$ is null-homotopic in $W$ need not extend to an embedding of a ball in $W$. However, if $W$ has abelian fundamental group with boundary a homology sphere, then $S$ bounds a topologically embedded ball in $W$. Additionally, we give examples where such an $S$ does not bound any smoothly embedded ball in $W$. In a similar vein, we construct incompressible tori $T\subseteq \partial W$ where $W$ is a contractible $4$-manifold such that $T$ extends to a map of a solid torus in $W$, but not to any embedding of a solid torus in $W$. Moreover, we construct an incompressible torus $T$ in the boundary of a contractible $4$-manifold $W$ such that $T$ extends to a topological embedding of a solid torus in $W$ but no smooth embedding. As an application of our results about tori, we address a question posed by Gompf about extending certain families of diffeomorphisms of $3$-manifolds; he has recently used such families to construct infinite order corks.  
\end{abstract}

\vspace{-.2in}

\section{Introduction}
The classical $3$-dimensional Dehn's lemma says that if an embedded circle in the boundary of a $3$-manifold is null-homotopic in the interior then it is in fact the boundary of an embedded disk. There are two possible analogues of this setup in the context of $4$-manifolds: we may ask about embedded circles in the boundary, or about embedded codimension one submanifolds of the boundary. The former is essentially a question about knot concordance (see \cite{livingston:survey,hom:survey} for excellent surveys of this active field). Here we consider the latter situation. Note that unlike the $3$-dimensional case, the distinction between the smooth and topological categories  for $4$-manifolds introduces an added element of complexity. In this paper, any topological embedding is assumed to be locally flat. 

%Given a manifold $Y^3$ that is the boundary of a compact manifold $W^4$, we consider an essential $2$-sphere $S\subseteq Y$ that is null-homotopic in $W$, and ask whether $S$ is the boundary of an embedded $3$-ball in $W$. The $4$-manifold $W$ will be either smooth or topological, as indicated below. We show that the most general $4$-dimensional analogue of Dehn's lemma for spheres does not hold, both for separating and non-separating spheres; however, the analogue does hold under broad hypotheses.
%
%
%\begin{T:introsphere} (Theorem~\ref{T:nodehnforsphere} and Corollary~\ref{C:top-ball-easy}) {\it A sphere $S$ in $Y^3=\partial W^4$ that is null-homotopic in $W$ need not bound a topologically embedded ball in $W$. More precisely, we give examples where $W$ is smooth and simply connected and $S$ separates $Y$; and when $W$ is topological and $S$ is non-separating. 
%
%However, if $Y$ is a homology sphere and $W$ is topological and has abelian fundamental group, any null-homotopic $S$ extends to a topologically embedded ball in $W$.}
%\end{T:introsphere}
%
%A more general sufficient condition is given in Theorem~\ref{T:top-ball-easy}. Moreover, we show that it is possible for a null-homotopic sphere to extend to a topologically embedded ball, but no smoothly embedded one. 

Given a manifold $Y^3$ that is the boundary of a compact manifold $W^4$, we consider an essential $2$-sphere $S\subseteq Y$ that is null-homotopic in $W$, and ask whether $S$ is the boundary of an embedded $3$-ball in $W$. The $4$-manifold $W$ will be either smooth or topological, as indicated below. We show that the most general $4$-dimensional analogue of Dehn's lemma for spheres does not hold, both for separating and non-separating spheres; however, the analogue does hold under broad hypotheses.

\begin{T:introsphere} {\it A sphere $S$ in $Y^3=\partial W^4$ that is null-homotopic in $W$ need not bound a topologically embedded ball in $W$. More precisely, there are such examples where $W$ is smooth and simply connected and $S$ separates $Y$; and when $W$ is topological and $S$ is non-separating. }
\end{T:introsphere}

In the converse direction, we give a general sufficient condition for a null-homotopic sphere $S$ to extend to a topologically embedded ball in Theorem~\ref{T:top-ball-easy}. As a special case, Corollary~\ref{C:top-ball-easy} gives such an extension result when $Y$ is a homology sphere, and $W$ is topological and has abelian fundamental group.  Moreover, we show that it is possible for a null-homotopic sphere to extend to a topologically embedded ball, but no smoothly embedded one. 

\begin{T:smooth-ball} {\it There exists a smooth simply connected $W$ with $Y=\partial W$ a homology sphere and $S\subseteq Y$ such that $S$ bounds a topologically embedded ball in $W$ but does not extend to a smoothly embedded ball in $W$.}
\end{T:smooth-ball}

%In Section~\ref{S:spheres}, we find examples where such a sphere does not bound a topologically embedded ball in $W$, showing that the most general analogue of Dehn's lemma for spheres does not hold. In Theorem~\ref{T:nodehnforsphere}, we demonstrate this failure both for when $S$ separates the boundary and when it is non-separating.  However, in some circumstances we do find an embedded ball with boundary $S$. For example, if $Y$ is a homology sphere and $W$ has abelian fundamental group, then we show in Corollary~\ref{C:top-ball-easy} that $S$ bounds an embedded topological ball in $W$; a more general result is given in Theorem~\ref{T:top-ball-easy}.  We also give examples in Theorem~\ref{T:smooth-ball} where $S$ bounds a topologically embedded ball, but no smoothly embedded ball. 

We also consider a separating incompressible torus $T\subseteq Y$ that extends to a map of a solid torus in $W$. Here we are following an analogy set forth in~\cite{scharlemann:loop}, i.e.\ the torus corresponds to a loop (as in the loop theorem) and the solid torus plays the role of a compressing disk. As before we show that the most general analogue of Dehn's lemma for tori does not hold. 

\begin{T:nodehnfortori} {\it There exists a contractible smooth $4$-manifold $W$ with an incompressible torus $T\subseteq Y$, such that $T$ extends to a map of a solid torus, but not to any topological embedding of a solid torus in $W$.}
\end{T:nodehnfortori}

%In Theorem~\ref{T:nodehnfortori}, we construct a family of examples where there is no topologically embedded solid torus in $W$ with boundary $T$. 
Analogous to the case for spheres, we give a sufficient condition for the analogue of Dehn's lemma to hold for tori in Proposition~\ref{P:embed}, and find cases where $T$ bounds a topologically embedded solid torus but no smoothly embedded solid torus.

\begin{T:torus-smooth} {\it There exists a contractible smooth $4$-manifold $W$ and a torus $T\subseteq Y$, where $T$ extends to a topological embedding of a solid torus in $W$, but not a smooth embedding.} \end{T:torus-smooth}

A recent result of Gompf~\cite{gompf:infinite-cork,gompf:infinite-cork-handlebody} about infinite order corks provides a different (but related) family of incompressible tori in the boundary of contractible 4-manifolds that do not bound smoothly embedded solid tori. This is based on constructing certain diffeomorphisms of the 3-manifold boundary using curves on the tori and studying whether these functions extend over the 4-manifold. We discuss this further in Section~\ref{S:diff} and specifically answer Question 1.6 from~\cite{gompf:infinite-cork} in the negative.

A theorem of Eliashberg~\cite{eliashberg:filling} (see~\cite[Theorem 16.3]{cieliebak-eliashberg:stein} for a detailed proof) says that any embedded 2-sphere in the boundary of a Stein manifold $W$ bounds a smoothly embedded ball in $W$. Thus our results on the failure of Dehn's lemma for spheres may be used to conclude that the 4-manifolds under consideration do not admit Stein structures. Eliashberg~\cite[Corollary 4.2]{eliashberg:filling} used the same argument to prove that $S^2 \times D^2$ admits no Stein structure; the difference here is that the $2$-spheres in our examples are null-homotopic. 

One may also wonder whether the corresponding result holds for tori, i.e.\ must a torus in the boundary of a Stein manifold $W$ bound a smoothly embedded solid torus in $W$? In this form it is easy to see that the answer is no, since $T^3$ is the boundary of $T^2\times D^2$ and thus most of the tori in $T^3$ do not bound solid tori. Moreover, our example in Theorem~\ref{T:torus-smooth} consists of a torus in the boundary of a Stein manifold $W$ which bounds a topologically embedded solid torus in $W$ but no smoothly embedded solid torus. 

\begin{acknowledgement}  
We thank Bob Gompf for pointing out the relevance of Eliashberg's work to our study of embedded spheres, and for the reference to~\cite{cieliebak-eliashberg:stein}. We also thank Peter Teichner for some comments related to the proof of Theorem~\ref{T:nodehnforsphere}\eqref{nonseparating}.
\end{acknowledgement}

%%%%%%%%%%%%%%%%%%%%%%%%%%%%%%%%%%%%%%%%%%%%%%%%%%%%%
\section{Dehn's lemma for spheres}\label{S:spheres}

Throughout this section, $Y$ denotes the $3$-manifold boundary of a compact $4$-manifold $W$. The $4$-manifold $W$ will be either smooth or merely topological, which we will indicate in any given case. We will use $S$ to denote a given essential embedded $2$-sphere in $Y$.  If $S$ is non-separating, then $Y$ splits as a connected sum of a $3$-manifold with a circle bundle over $S$ via a standard $3$-dimensional argument (see e.g. \cite[Lemma~3.8]{hempel:3manifolds}).  Otherwise $S$ will be null-homologous, in which case $Y = Y_1 \#_S Y_2$, for some $3$-manifolds $Y_1$ and $Y_2$, where the notation indicates that the connected-sum is performed along the sphere $S$.  For any closed manifold $M^3$, we denote by $M^\circ$ the punctured manifold $M -\inte(B^3)$. 

We start with two straightforward lemmata, giving a sufficient condition for $S$ to be null-homotopic in $W$ and a consequence of a null-homologous sphere in the boundary of a $4$-manifold bounding an embedded ball.

\begin{lemma}\label{L:spherenullhomotopic} Suppose that $W$ is simply connected and $S$ is null-homologous in $Y$. Then $S$ is null-homotopic in $W$.\end{lemma}

\begin{proof}Since $S$ is trivial in $H_2(Y)$, it is also trivial in $H_2(W)$. However, $\pi_2(W)\cong H_2(W)$ by the Hurewicz theorem since $W$ is simply connected. \end{proof}

\begin{lemma}\label{L:boundaryconnectedsum} If $B$ is a properly embedded ball in a simply connected $4$-manifold $W$ then $W$ splits as a boundary-connected sum along $B$.\end{lemma}

\begin{proof} It suffices to show that $W - (B\times [0,1])$ has two connected components. We use the Mayer--Vietoris sequence for $W=W - (B\times [0,1]) \cup (B\times [0,1])$: 
\[
\xymatrix{
H_1(W)\ar[r] &\widetilde{H_0}(B\times \{0,1\}) \ar[r] &\widetilde{H_0}(B\times [0,1])\oplus \widetilde{H_0}(W-B\times [0,1])\ar[r] &\widetilde{H_0}(W)
}
\]
which is simply 
\[
\xymatrix{
0\ar[r]&\Z \ar[r] & \widetilde{H_0}(W-B\times [0,1])\ar[r] & 0
}\hfill %\qedhere
\]
\end{proof}

The following theorem shows that Dehn's lemma does not always hold in dimension $4$, in both the separating and non-separating cases. In each setting the arguments rely on properties of the intersection form.

\begin{theorem}\label{T:nodehnforsphere} A sphere in the boundary of a $4$-manifold $W$ that is null-homotopic in $W$,  need not bound a ball in $W$. More precisely,
\begin{enumerate}
\item\label{separating} There exists a sphere $S$ and a $4$-manifold $W$ such that $W$ is smooth and simply connected, $S$ is null-homotopic in $W$ and separating in $Y=\partial W$, but $S$ does not bound an embedded topological ball in $W$.
\item\label{nonseparating} There is a topological $4$-manifold $W$ with boundary $S^1 \times S^2$ such that the essential non-separating $2$-sphere $S \subset S^1 \times S^2$ is null-homotopic in $W$, but does not bound an embedded $3$-ball in $W$.
\end{enumerate}
\end{theorem}
\begin{proof} For the first part, let $L$ be a lens space $L(p,q)$, where $\pm q$ is not a quadratic residue mod $p$, for example $L(5,2)$. Let $\gamma$ be a curve generating $\pi_1(L)$. Let $W$ be the $4$-manifold obtained from $L^\circ\times [0,1]$ by performing surgery on $\gamma$ pushed into the interior. By construction $W$ is simply connected with $Y=\partial W=-L\# L$ where the connected-sum is performed along a sphere $S$. Since $S$ is null-homologous in $Y$, $S$ is null-homotopic in $W$ by Lemma~\ref{L:spherenullhomotopic}. We compute that $H_2(W)=\Z\oplus \Z$.

Suppose that $S$ bounds an embedded topological ball in $W$. Then by Lemma~\ref{L:boundaryconnectedsum}, $W$ decomposes along the ball as a boundary-connected sum $W_1\natural W_2$. Since $W$ is simply connected, each $W_i$ is simply connected, and since the ball is bounded by $S$, $\partial W_1=-L$ and $\partial W_2=L$. From the following piece of the long exact sequence of the pair $(W_i,\partial W_i)$:
\[
\xymatrix{
0\ar[r]&H_2(W_i) \ar[r] & H_2(W_i,\partial W_i)\ar[r] & H_1(\partial W_i) \ar[r] & 0
}\hfill 
\]
we see that $H_2(W_i,\partial W_i)$ surjects onto $\Z/p\Z$. By Poincar\'{e}--Lefschetz duality and the universal coefficient theorem, and the fact that $H_1(W_i)=0$, $H_2(W_i,\partial W_i)$ is torsion-free. Thus, we cannot have $H_2(W_i)=0$ for either $i=1,2$. Since $H_2(W)=H_2(W_1)\oplus H_2(W_2)=\Z\oplus \Z$, this implies that $H_2(W_1)\cong H_2(W_2)\cong \Z$. However by~\cite[Proposition 2.2, Example 2.3]{Saeki89}, 
we know that a lens space bounds a simply connected topological $4$-manifold with $b_2=1$ if and only if $\pm q$ is a quadratic residue mod $p$. Briefly, the argument consists of showing that the intersection form of such a $4$-manifold is represented by the matrix $(p)$ while the linking form on the lens space is represented by $(\pm q/p)$.

The proof of the second part is based on a construction of Hambleton and Teichner~\cite{hambleton-teichner:non-extended}, who find a closed topological manifold $X$ with $\pi_1(X) \cong \Z$, where the $\Z[\Z]$ intersection form is not induced from a unimodular form over $\Z$ by tensoring with 
$\Z[\Z]$.  In fact, since their form (over $\Z$) is odd, they find two such manifolds, distinguished up to homeomorphism by the Kirby-Siebenmann invariant. Letting $X$ denote either one, then our example $W$ is simply $X - S^1 \times \inte(B^3)$, where the circle generates $\pi_1(X)$.  If $S = \partial B^3$ bounds a disk, then the generator of $H_3(X)$ would be represented by a $3$-sphere, exhibiting $X$ as a connected sum of $S^1\times S^3$ with a closed simply connected $4$-manifold, and contradicting the property of the intersection form.

To verify that $S$ is null-homotopic, we go through a construction of a manifold with the same $\Z[\Z]$ intersection form as the one in~\cite{hambleton-teichner:non-extended}, remarking at the end that we could get either value of the Kirby-Siebenmann invariant. Hambleton and Teichner appeal to a general existence theorem of Freedman-Quinn~\cite[Theorem 10.7]{freedman-quinn} that says that an arbitrary unimodular Hermitian form over $\Z[\Z]$ is realized by a topological $4$-manifold. This proof constructs a stabilized version of the form, and then removes hyperbolic summands by surgery; it is not easy to track a given homology class in the construction. So we give an alternate version that will help in verifying that $S$ is null-homotopic. 

Suppose that we have a unimodular Hermitian form on a free module over $\Z[\Z]$, given by an $n \times n$ matrix $A$ with coefficients in $\Z[\Z]$, and construct $W$ in two steps.  We choose a generator $t$ of the fundamental group of $S^2 \times S^1$, and identify $\Z[\Z]$ with $\Z[t,t^{-1}]$.  We also treat $t$ as a generator of the covering transformations of the infinite cyclic cover of $S^2 \times S^1$ (and any other manifold that might turn up in this discussion.) First we consider 
\begin{equation}\label{E:W}
W_0 = \left(S^2 \times S^1 \times I\right) \cup \bigcup_{i= 1}^n h_i^2
\end{equation}
where the $h_i^2$ are $2$-handles attached along null-homologous curves $\gamma_i$ in $S^2 \times S^1\times \{1\}$. The attaching curves are organized so that they link one another according to the matrix $A$. More precisely, since each $\gamma_i$ is null-homologous, we can fix a lift $\gtilde_i$ to the universal cover $S^2 \times \R$. The curves will have the property that for $i \neq j$
$$
\sum \lk(\gtilde_i, t^k(\gtilde_j)) t^k
$$
is the $ij$ entry in $A$.

The framings of the $\gamma_i$ are determined so that if $\gtilde_i'$ denotes the lift of a parallel to $\gamma_i$, then
$$
\sum \lk(\gtilde_i', t^k(\gtilde_i)) t^k
$$
is the $ii$ entry in $A$.  In order to construct such a link, we start with a trivial $n$-component link, and do finger moves that go around the $S^1$ direction as needed. For example, if the $ii$ entry is $mt^{-2} + k +  mt^2$ the attaching circle of $h_i$ would be as indicated in Figure \ref{F:finger}, with framing $k+2m$. 
\begin{figure}[h]
  \centering
    \labellist
   \small\hair 2pt
 \pinlabel {$m$}  at 248 345
 \pinlabel {$k+2m$}  at 248 50
\endlabellist
  \includegraphics[scale = .4]{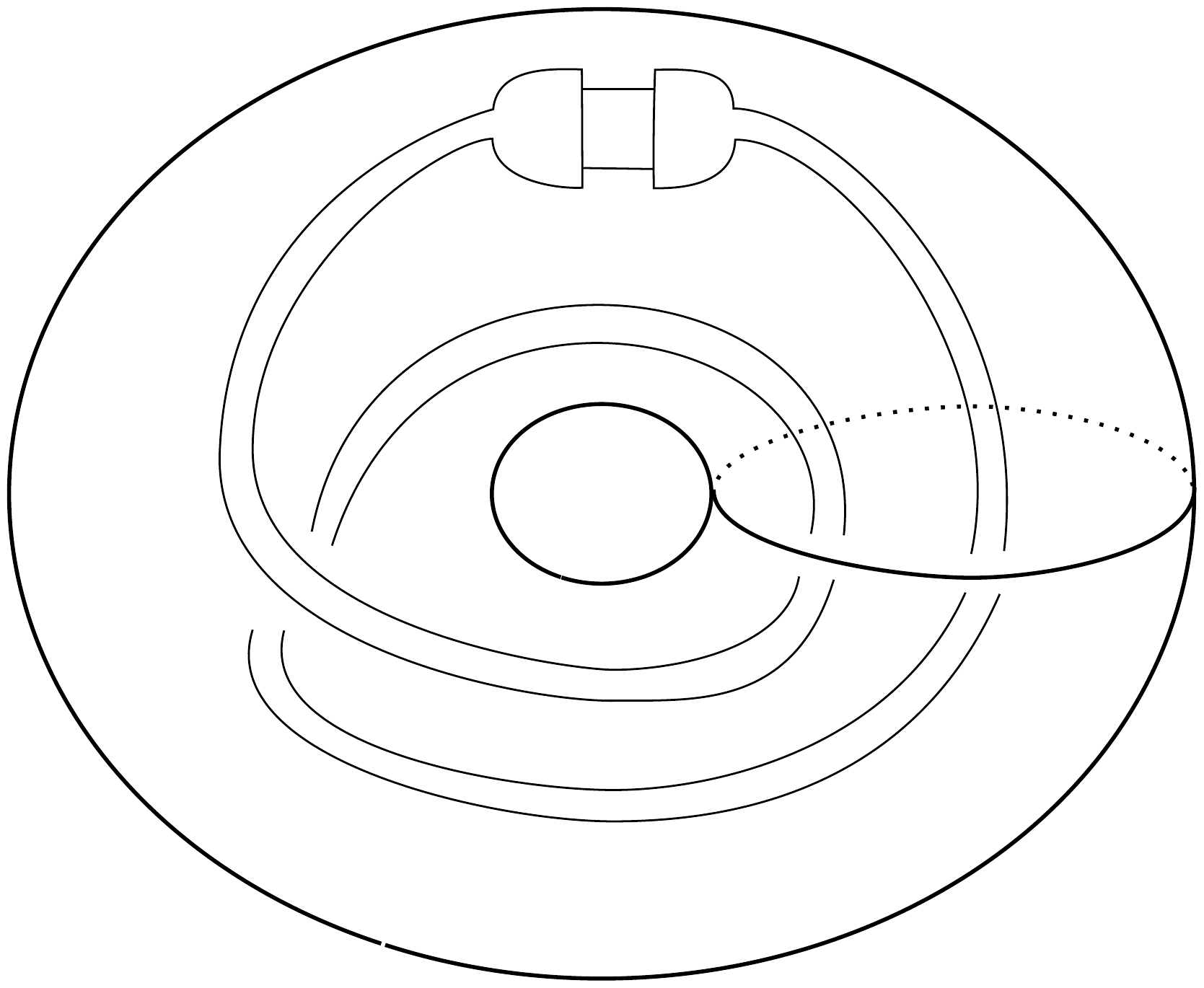}
  \caption{Example surgery in proof of Theorem~\ref{T:nodehnforsphere}}\label{F:finger}
\end{figure}
This is essentially the same argument as Wall~\cite[Theorem 5.8]{wall:book} uses for realization of the action of $L_4(\Z)$ on the structure set of a $(4k-1)$-manifold with $\pi_1 = \Z$ (the homology structure set if $k=1$), or the picture in \cite[page 344]{rolfsen:knots}. Compare the construction in~\cite[\S 4]{friedl-hambleton-melvin-teichner}, which is based on the same basic idea.

The boundary of $W_0$ consists of the original copy of $S^2 \times S^1$, together with a manifold $Y$. Because the matrix $A$ is unimodular, it follows that $Y$ has the same $\Z[\Z]$ homology as $S^2 \times S^1$, i.e.\ $H_1(Y; \Z[\Z])$ is trivial. The proof of \cite[Theorem 11.7B]{freedman-quinn}, showing that Alexander polynomial one knots are slice, also shows that $Y = \partial V$, where $V$ has the homotopy type of a circle. Then we set $W = W_0 \cup_Y V$.

To show that $S$ is null-homotopic in $W$, we lift it to a sphere $\St \subset \partial \wtilde$, and show that $\St$ is null-homologous in $\wtilde$. Since $\wtilde$ is simply connected, Lemma~\ref{L:spherenullhomotopic} will imply that $\St$ is null-homotopic.  
Note that as a $\Z[\Z]$ module, 
$$
H_2(\wtilde_0) \cong \Z \oplus \left(\Z[\Z]\right)^n
$$
where the action of $\Z[\Z]$ on the first summand is via the augmentation $\Z[\Z] \to \Z$. This is readily seen by lifting the handle decomposition \eqref{E:W} to the universal cover, for each lift of a $2$-handle $h_i^2$ is attached along a null-homologous curve in $S^2 \times \R$.  In particular, the first summand is generated by the homology class of $\St$.

Now consider the Mayer-Vietoris sequence for the decomposition 
$$
\wtilde = \wtilde_0 \cup_{\ytilde} \vtilde
$$
which looks like  
$$
\xymatrix{
&&{}\save[]+<2.5cm,.5cm>*\txt<3pc>{%
\hspace*{1.5ex}
\raisebox{.2cm}{$H_2(\St)$}}
\ar[d]_{i_S}\ar[dr] \restore& \\
H_3(\wtilde) \ar[r] \ar@{=}[d]& H_2(\ytilde) \ar@{=}[d]\ar[r]^-{(i,j)} & H_2(\vtilde)\bigoplus H_2(\wtilde_0) \ar@{=}[d]\ar[r]^-k & H_2(\wtilde)\ar@{=}[d]\\
0 \ar[r] & \Z \ar[r]^-{(i,j)} &0  \bigoplus  (\Z \oplus \left(\Z[\Z]\right)^n) \ar[r]^-k   & H_2(\wtilde)
}
$$
Above, $H_3(\widetilde{W})=0$ since $W_0$ is built from $S^2\times \R$ by adding $2$-handles and $V$ is a homotopy circle. Any $\Z[\Z]$-module map from $\Z$ to $\left(\Z[\Z]\right)^n$ must be trivial, thus the homomorphism $j$ in the diagram must have its image in the $\Z$ summand of $H_2(\wtilde_0)$; similarly the image of $i_S$ lies in that same summand. By exactness, the restriction of $k$ to $H_2(\wtilde_0)$, composed with $j$, is $0$. But this means that this same restriction, composed with $i_S$, is $0$.  This last composition is exactly the map induced by the inclusion of $\St$ into $\wtilde$. Thus, $\widetilde{S}=0\in H_2(\widetilde{W})=\pi_2(W)$ as claimed. 

When done with care, as in~\cite[\S 4]{friedl-hambleton-melvin-teichner}, the construction above yields a manifold $W$ with vanishing Kirby-Siebenmann invariant (relative to $S^1 \times S^2$.) To get a manifold with non-vanishing Kirby-Siebenmann invariant, we proceed as in the proof of~\cite[Theorem 1.5]{freedman:simply-connected}. Since the ordinary $\Z$-valued intersection form is odd, we can do handle slides so that a characteristic element in the $\Z_2$ homology is represented by a single $2$-handle. As in~\cite{freedman:simply-connected}, changing the attaching map of this handle by connected sum with a trefoil knot will change the Kirby-Siebenmann invariant.
\end{proof}

We remark that the manifolds used in the second part of the proof above are known to not be smoothable~\cite{friedl-hambleton-melvin-teichner}.  A smooth example would come from a smoothable $4$-manifold with fundamental group $\Z$ whose $\Z[\Z]$ intersection form is not extended from the integers; as far as we know no such manifold has been constructed.

Theorem~\ref{T:nodehnforsphere} shows that an analogue for Dehn's lemma does not hold for spheres in the boundary of some $4$-manifolds. However, the statement does hold if we restrict to the situation where $W$ is simply connected and $Y$ is a homology sphere. Indeed, the following more general result is true. 

\begin{theorem}\label{T:top-ball-easy} Suppose that $Y =Y_1 \conn_S Y_2$, where $Y_2$ is a homology sphere and $Y = \partial W$, such that
\begin{enumerate}
\item $\pi_1(W)$ is `good', and
\item the induced map $\pi_1(Y_2) \to \pi_1(W)$ is trivial.
\end{enumerate}
Then there is a topological embedding of $B^3$ in $W$ with boundary $S$.\end{theorem} 

Above, $\pi_1(W)$ is `good' in the sense of Freedman, i.e.\ the surgery sequence in the topological category is exact for $W$ and the $s$-cobordism theorem holds for $5$-dimensional $s$-cobordisms with fundamental group $\pi_1(W)$.
%\begin{proof*}[of Theorem~{\rm\ref{T:top-ball-easy}}]
\begin{proof}[Proof of Theorem~\ref{T:top-ball-easy}]
We recall that Freedman~\cite{freedman-quinn} showed that any homology sphere is the boundary of a contractible manifold; hence there is a simply connected integral homology cobordism, relative to the boundary, from $Y_2^\circ$ to $B^3$. Double this cobordism along $B^3$ to get a simply connected integral homology cobordism, say $Z$, from $Y_2^\circ$ to itself.  The important thing to note is that $S =\partial Y_2^\circ$ bounds a $3$-ball in $Z$, drawn in Figure~\ref{F:plug}.  The vertical part of the figure is $S \times I$.

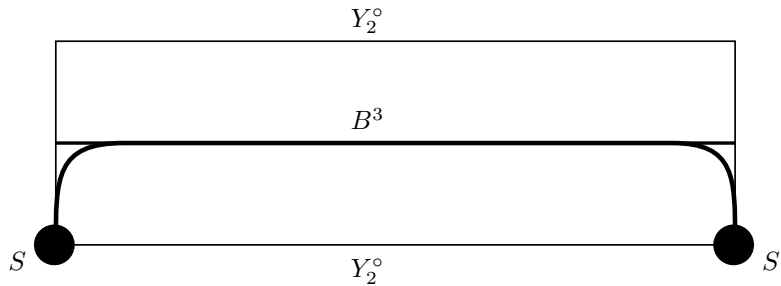
\begin{figure}[htb]
\vspace*{2ex}
\labellist
\small\hair 0mm
  \pinlabel {$Y_2^\circ$}  at 160 118
 \pinlabel {$B^3$}  at 160 72
 \pinlabel {$Y_2^\circ$}  at 160 0
 %\pinlabel {$Z$}  at 80 100
 \pinlabel {$S$}  at -5 5
 \pinlabel {$S$}  at 350 5
\endlabellist
\centering
\includegraphics[scale=.8]{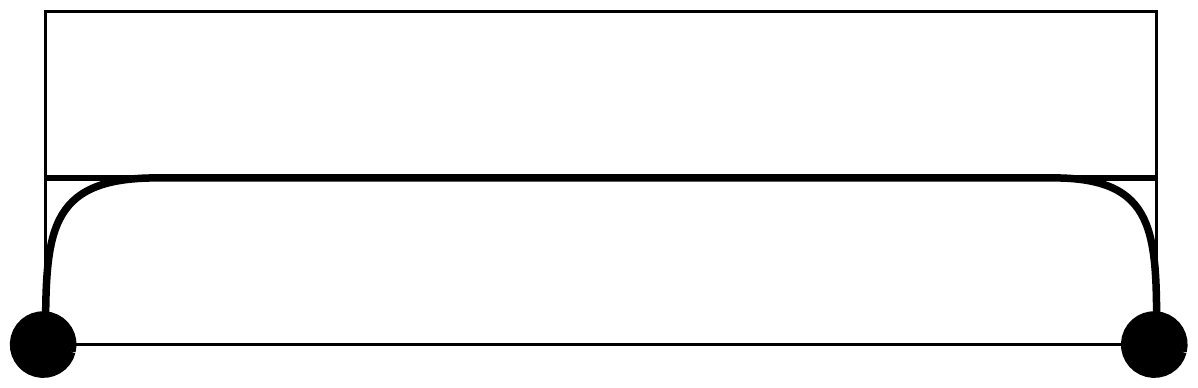}
\caption{The homology cobordism $Z$}
\label{F:plug}
\end{figure}

We construct a manifold $W'$ with boundary $Y$ by first adding a collar $Y_1^\circ \times I$ to $W$ along $Y_1^\circ \times \{0\}$ and then gluing on a copy of the manifold $Z$ constructed above; see Figure~\ref{F:Wprime}. By our hypothesis on fundamental groups, $\pi_1(W') \cong \pi_1(W)$.
\begin{figure}[htb]
\vspace*{2ex}
\labellist
\small\hair 0mm
 \pinlabel {$S$}  at 38 70
 \pinlabel {$S$}  at 380 70
 \pinlabel {$Y_1^\circ \times I$}  at 200  390
 \pinlabel {$W$}  at 200 200
\pinlabel {$Y_2^\circ$}  at 200 90
 \pinlabel {$Z$}  at 200 20
\endlabellist
\centering
\includegraphics[scale=.5]{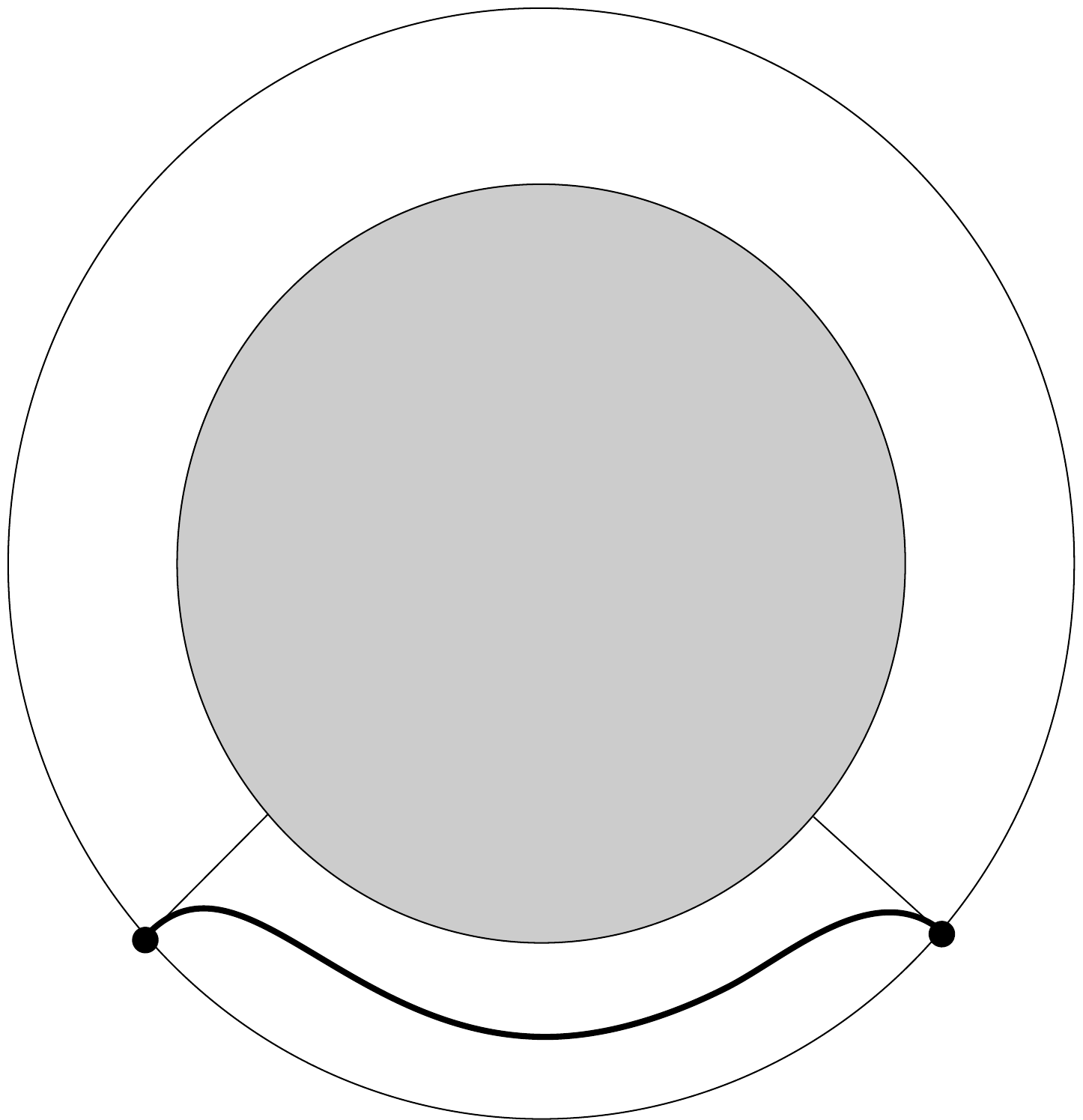}
\caption{The manifold $W'$}
\label{F:Wprime}
\end{figure}
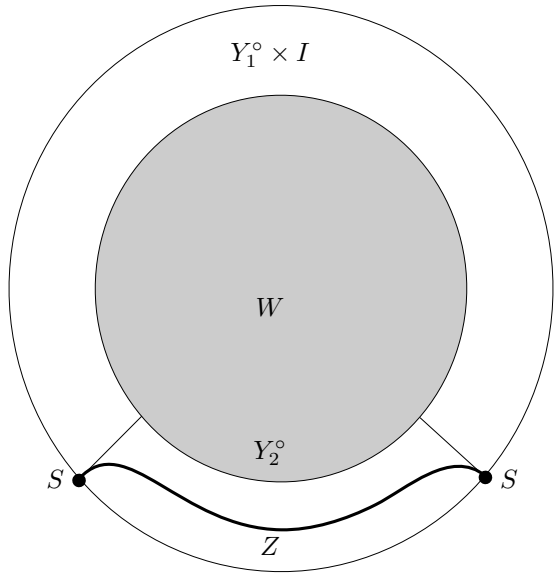

We will show that $W'$ is homeomorphic to $W$ by constructing an $s$-cobordism between them that is a product along the boundary; our approach is similar to that in~\cite[Lemma~3.2]{kreck-luck-teichner:kneser}. Consider the union of $Z$ and $Y_2^\circ \times I$ glued along their boundary. This is a homology $4$-sphere, and as such~\cite[Theorem 3]{kervaire:homology} it bounds a contractible $5$-manifold $A$. We can view $A$ as a homology cobordism between $Y_2^\circ \times I$ and $Z$, with a product structure along $\partial(Y_2^\circ\times I)=Y_2\# -Y_2$.
 
Now we construct a cobordism  
$$
U = \left(W \cup_{Y^\circ_1} Y^\circ_1 \times I \right) \times I \cup_{(Y_2^\circ \times \{0\}\cup S\times I )\times I} A
$$ 
between $W$ and $W'$.  This is simpler than it sounds; see Figure~\ref{F:s-cobordism}.
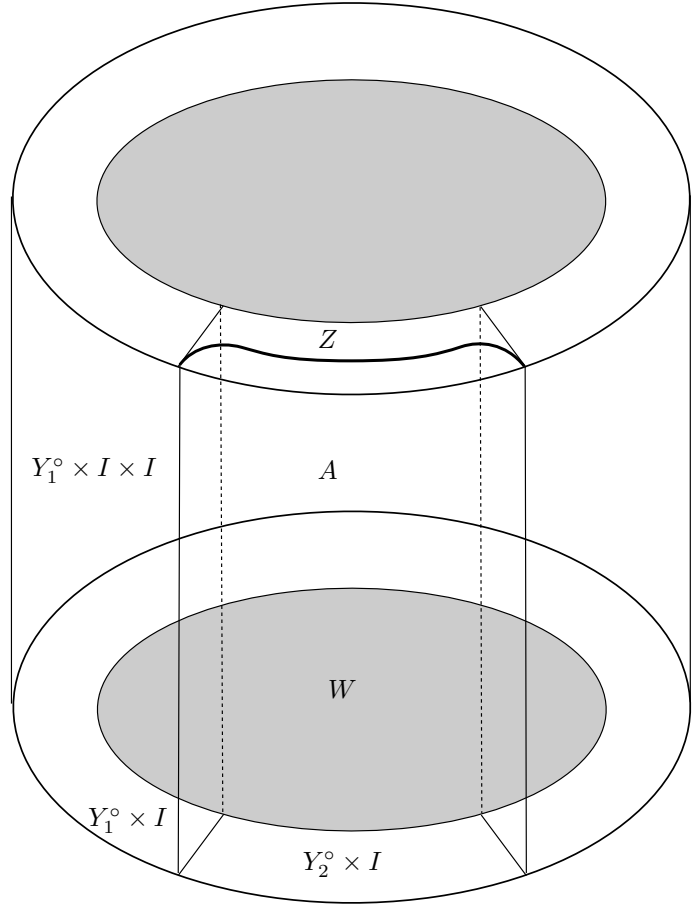
\begin{figure}[htb]
\vspace*{2ex}
\labellist
\small\hair 0mm
 \pinlabel {$Y_2^\circ \times I$}  at 230 30
 \pinlabel {$Y_1^\circ \times I$}  at 83 60
 \pinlabel {$Z$}  at 220  390
 \pinlabel {$W$}  at 230 150
 \pinlabel {$Y_1^\circ \times I \times I$}  at 60 300
 \pinlabel {$A$}  at 220 300
\endlabellist
\centering
\includegraphics[scale=.55]{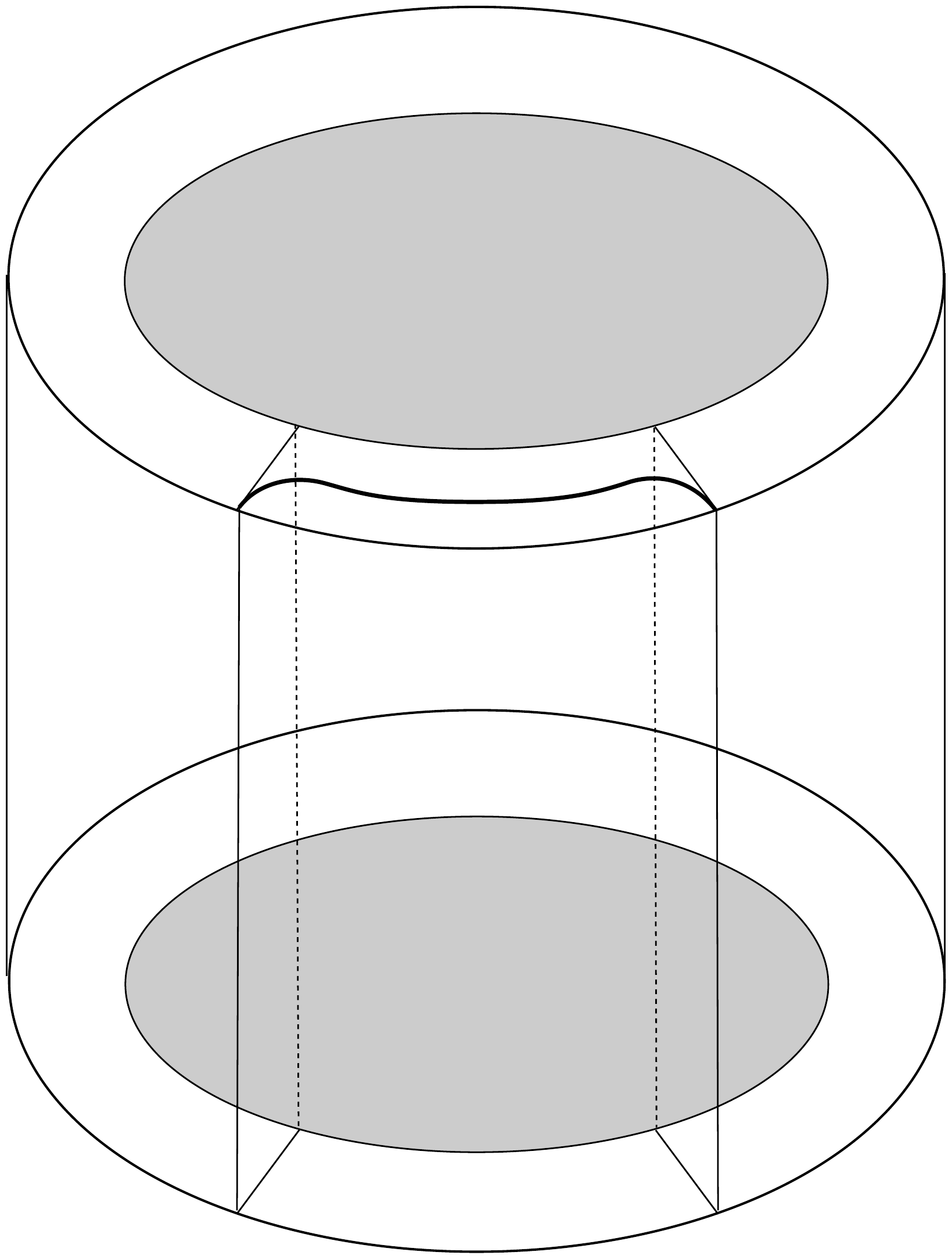}
\caption{The $s$-cobordism $U$}
\label{F:s-cobordism}
\end{figure}

A Mayer-Vietoris argument implies that $U$ is an integral homology cobordism (rel boundary) between $W$ and $W'$, and the hypothesis on fundamental groups implies (via van Kampen's theorem) that the inclusion maps from $W$ and $W'$ induce isomorphisms between their fundamental groups and $\pi_1(U)$. Next we show that $U$ is in fact an $s$--cobordism by examining the chain complex of $U$ with $\Z[\pi_1(U)]$ coefficients. We remind the reader that in this setting, the torsion for a topological $h$-cobordism may be defined
in terms of the (based) chain complex associated to a handle decomposition~\cite[Theorem 2.3.1]{Quinn:ends} of $U$ relative to $W$.

Let $B$ denote the product $(W\cup Y_1^\circ\times I)\times I$. Then $U=A\cup B$ with $A\cap B=(Y_2^\circ\times \{0\}\cup S\times I)\times I$. Let $\partial_0 U$ denote $W\cup \partial W\times I$, i.e.\ the `base' of the cobordism. Let $\partial_0 A=Y_2^\circ \times I$, $\partial_0 B=W\cup Y_1^\circ\times I$. 

By~\cite[Theorem 2.3.1]{Quinn:ends}, we get a (topological) handle decomposition for $A$ relative to $\partial_0 A $. Along with the product structure on $B$, this gives a handle decomposition for $U$ relative to $\partial_0 U$, all of whose handles lie in $A$ with attaching maps in $\partial_0 A$.

We are interested in the chain complex $C(\widetilde{U},\widetilde{\partial_0 U})$ of $\Z[\pi_1(\partial_0U)]$--modules, where $\widetilde{U}$ and $\widetilde{\partial_0 U}$ denote universal covers. Here and henceforth, the chain complex comes from the relative handle decomposition constructed above. Let $\widetilde{A}$ and $\widetilde{B}$ denote the induced covers of $A$ and $B$, respectively. Because all of the handles of $U$ lie in $A$ and $A$ is simply connected, the chain complexes (over $\Z[\pi_1(\partial_0 U)]$)
$$
C(\widetilde{A},\widetilde{\partial_0 A})\ \text{and}\ C(\widetilde{U},\widetilde{\partial_0 U})
$$
are identical.  In particular, 
\begin{equation}\label{E:complex}
C(\widetilde{A},\widetilde{\partial_0 A}) \cong C(A,\partial_0 A) \otimes_{\Z} \Z[\pi_1(\partial_0 U)].
\end{equation}
Since $A$ is a $\Z$-homology cobordism, $C(\widetilde{U},\widetilde{\partial_0 U})$ is also acyclic as a complex with coefficients in $\Z[\pi_1(\partial_0 U)]$, and it follows from Whitehead's theorem that $U$ is an $h$-cobordism.

Now, the Whitehead torsion $\tau(U,\partial_0 U)$ is defined to be the torsion of the chain complex $C(\widetilde{U},\widetilde{\partial_0 U})$, and \eqref{E:complex} implies that $\tau(U,\partial_0 U) $ is the same as the torsion of the homology cobordism $(A,\partial_0 A)$.  Hence 
$$
\tau(U,\partial_0 U)  = \rho_*(\tau(A,\partial_0 A))\ \text{where}\ 
\rho_*: \Wh(\{1\}) \to \Wh(\pi_1(\partial_0 U))
$$
is induced by the inclusion of the trivial group. Since $\Wh(\{1\})$ is trivial, $\tau(U,\partial_0 U)$ vanishes, so $U$ is an $s$-cobordism.
 
By the relative $s$-cobordism theorem, $U$ has a topological product structure extending the given one on the boundary. In particular, $W'$ is homeomorphic to $\partial_0 U \cong W$. Since $W'$ contains a topological ball with boundary $S$, it follows that $W$ contains such a ball.
\end{proof}

We note the following special case of the above theorem.
\begin{corollary}\label{C:top-ball-easy} If $Y=Y_1\#_S Y_2$ is a homology sphere and $\pi_1(W)$ is abelian (e.g.\ if $W$ is simply connected), there is an embedding of $B^3$ in $W$ with boundary $S$.
\end{corollary} 
\begin{proof}
Since $Y$ is a homology sphere $S$ is null-homologous. In~\cite{freedman:goodgroups}, Freedman showed that abelian groups, including the trivial group, are `good'. The remaining hypothesis of Theorem~\ref{T:top-ball-easy} holds if $\pi_1(W)$ is abelian since any homomorphism from a perfect group to an abelian group is trivial. \end{proof}

In contrast, we have the following result in the smooth category. 

\begin{theorem}\label{T:smooth-ball} There exists a smooth simply connected $W$ with $Y=\partial W$ a homology sphere and $S\subseteq Y$ such that $S$ bounds a topologically embedded ball in $W$ but does not extend to a smoothly embedded ball in $W$.
\end{theorem} 
\begin{proof}We will build a simply connected spin 4-manifold $W$ with $\partial W=Y_1\conn_S Y_2$ where the $Y_i$ are oriented homology spheres. By Corollary~\ref{C:top-ball-easy}, $S$ bounds a topologically embedded ball in $W$. If $S$ bounds a smooth ball in $W$, by Lemma~\ref{L:boundaryconnectedsum}, $W$ decomposes along this ball as a boundary-connected sum $W=W_1\natural W_2$, where each $W_i$ is smooth, spin, and simply connected, with $Y_i=\partial W_i$, and moreover the intersection form splits as $Q_W\cong Q_{W_1}\oplus Q_{W_2}$. Thus, we can obstruct $S$ from bounding a smooth embedded ball in $W$ by choosing the $Q_W$ and the $Y_i$ wisely. We show how to proceed in the two specific cases below, and then construct concrete examples. \begin{enumerate}
\item $Q_W\cong H$, and $\rho(Y_1)=1$,
\item $Q_W\cong E_8\oplus H$, $\rho(Y_1)=1$, $\rho(Y_2)=0$, $Y_1$ is not the boundary of any smooth 4-manifold with intersection form $E_8$, and $Y_2$ is not the boundary of any acyclic smooth 4--manifold.
\end{enumerate}
In the first case, since $W=W_1\natural W_2$ where each $W_i$ is spin and simply connected with boundary a homology sphere, we see that $Q_{W_1}$ is either $0$ or $H$. This is a contradiction since $\rho(Y_1)=1$. In the second case, by the values of the Rohlin invariants, we see that the signature of $Q_{W_1}$ must be $8$ and the signature of $Q_{W_2}$ must be $0$. By the classification of indefinite even unimodular forms~\cite{milnor-husemoller}, there are two possibilities; either $Q_{W_1}\cong E_8$ and $Q_{W_2}\cong H$, or $Q_{W_1}\cong E_8\oplus H$ and $Q_{W_2}\cong 0$. Both possibilities are ruled out by hypothesis.

We finish the proof by giving concrete examples of manifolds satisfying the two situations mentioned above. Let $P$ denote the Poincar\'{e} homology sphere and $\gamma$ be a curve that normally generates $\pi_1(P)$. For the first case, let $W$ denote the $4$-manifold obtained by performing surgery on $P^\circ\times [0,1]$ along $\gamma$ pushed into the interior, with framing such that $W$ is spin. Then $W$ has the hyperbolic intersection form $H$ and is simply connected by construction with $\partial W=-P \# P$, where the connected-sum is along a sphere $S$. It is well-known that $\rho(\pm P)=1$. 

For the second case, let $Y_1$ be $P$ oriented as the boundary of $X$, the negative definite plumbing with form $E_8$.  Thus, $\rho(Y_1)=1$. If $Y_1$ were the boundary of a smooth $4$-manifold with intersection form $E_8$, we could construct a negative definite $4$-manifold with intersection form $E_8 \oplus E_8$, contradicting Donaldson's theorem.  Let $Y_2=P \conn P$, which is well-known to not bound an acyclic manifold~\cite{fs:pseudo}. Additionally, $\rho(Y_2)=0$. We construct $W$ as follows. Start with the boundary-connected sum of $X$ with $P^\circ \times I$. Let $S$ denote the boundary of $P^\circ$. Then the boundary of this manifold is $P \conn -P \conn P=-P \conn (P \conn P) = Y_1 \conn Y_2$. Now do surgery along $\gamma$ in $P^\circ$, pushed into the interior, with framing such that the resulting manifold $W$ is spin. By construction, $W$ is simply connected, and its intersection form is $E_8 \oplus H$, where $H$ is the hyperbolic intersection form; compare~\cite{scharlemann:phs}. \end{proof}

We note that while we focused on two specific cases in the above proof, the general principle likely applies to other choices of intersection forms and Rohlin invariants. 

A non-separating sphere which bounds a topological ball but no smooth ball is provided by Fintushel and Stern in~\cite{fs:fake}. They find a smooth $4$-manifold homeomorphic to $K3 \conn (\sss) \conn (S^1 \times S^3)$ that does not smoothly split off a copy of $S^1 \times S^3$. Removing a neighborhood of a circle generating the fundamental group gives rise to a smooth $4$-manifold with boundary $S^1 \times S^2$, where the $2$-sphere is null homotopic but does not bound a smooth $3$-ball. However, it does bound a topologically embedded $3$-ball. 

%%%%%%%%%%%%%%%%%%%%%%%%%%%%%%%%%%%%%%%%%%%%%%%%%%%%%
\section{Dehn's lemma for tori}\label{S:tori}
As before, throughout this section, $Y$ denotes the $3$-manifold boundary of a compact $4$-manifold $W$. We will use $T$ to denote an embedded torus in $Y$. 

Whenever we have a sphere in a 3--manifold, we may add a trivial handle to it to obtain a torus. This allows us to easily construct (compressible) tori which do not satisfy Dehn's lemma using our examples from the previous section. For instance, suppose that we start with a sphere $S$ in $Y$ which is nullhomotopic in $W$ but does not bound an embedded ball in $W$. Perform a trivial stabilization to obtain a torus $T$. Note that there are dual curves $\alpha$ and $\beta$ in $T$ both of which bound embedded disks in $Y$. Suppose such a $T$ bounds an embedded solid torus in $W$. Then either $\alpha$ or $\beta$ bounds an embedded disk in $W$. Use the disk in $Y$ bounded by the dual curve to obtain an embedded ball in $W$ bounded by a sphere isotopic to $S$, which is a contradiction. 

In order to avoid the above situation, we will restrict ourselves to incompressible tori in the rest of this section. We start by showing that an analogue for Dehn's lemma does not hold for incompressible tori in the boundary of contractible $4$-manifolds. 

\begin{theorem}\label{T:nodehnfortori} There exists a contractible $4$-manifold $W$ with an incompressible torus $T\subseteq Y$, such that $T$ extends to a map of a solid torus, but not to any topological embedding of a solid torus. 
\end{theorem}

\begin{proof}
Let $W$ denote the $4$-manifold given by the handle diagram in Figure~\ref{F:generaltoruslooptheorem}; this is inspired by~\cite[Figure 2]{akbulut:infinite} (reproduced in Figure~\ref{F:toruslooptheorem}). The dotted ribbon knot notation is described in detail in~\cite[p 213]{GS99}; see Figures~\ref{F:nosmoothsolidtorus1} and \ref{F:nosmoothsolidtorus2} for diagrams using the more common notation allowed only dotted unknots, in the specific case where $J$ is the left-handed trefoil. Briefly, the dotted ribbon knot $J\# -J$ in Figure~\ref{F:generaltoruslooptheorem} indicates that we take the complement of the canonical ribbon disk for $J\# -J$.

\begin{figure}[htb]
  \centering
    \labellist
    \normalsize\hair 0mm
    \pinlabel {$J$} at 190 55
    \pinlabel {$-J$} at 25 55
    \pinlabel {$K$} at 132 132
    \pinlabel {$n$} at 125 162
    \pinlabel {$T$} at 30 100
		\pinlabel {$T'$} at 150 152
    \pinlabel {$\alpha$} at 83 107
    \pinlabel {$\beta$} at 60 80    
    \pinlabel {$L_1$} at 30 25
    \pinlabel {$L_2$} at 112 135    
      \endlabellist
  \includegraphics[scale = 1]{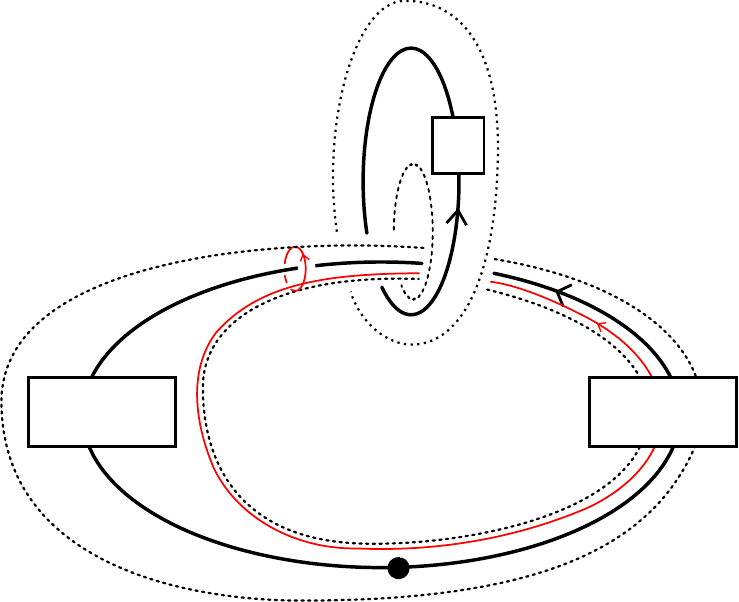}
  \caption{A torus $T$ (shown with dotted lines) embedded in the boundary of a contractible $4$-manifold $W$, which is given by the link $L_1 \sqcup L_2$. Here $n$ is an integer and $K$ and $J$ are knots. Choose $J$ to be a non-trivial knot. A strand passing vertically through a box labeled by a knot has that knot type; multiple curves passing through a box are 0--framed parallels. Two curves on the torus $T$, a meridian $\alpha$ and a longitude $\beta$ are shown in red. The torus $T'$ is also shown with dotted lines.}\label{F:generaltoruslooptheorem}
\end{figure}
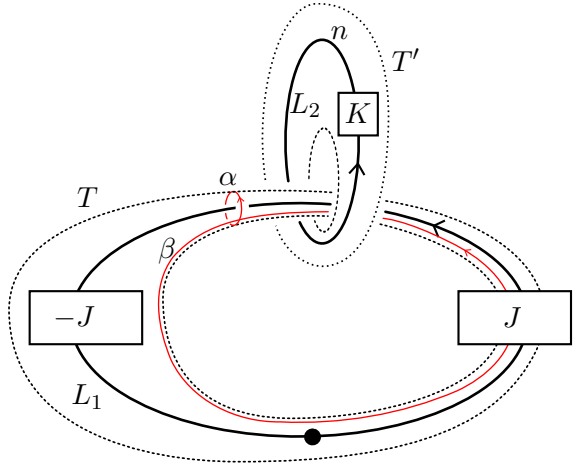
\begin{figure}[htb]
  \centering
    \labellist
    \normalsize\hair 0mm
    \pinlabel {$-1$} at 202 180
    \pinlabel {$T$} at 100 137
      \endlabellist
  \includegraphics[scale = .8]{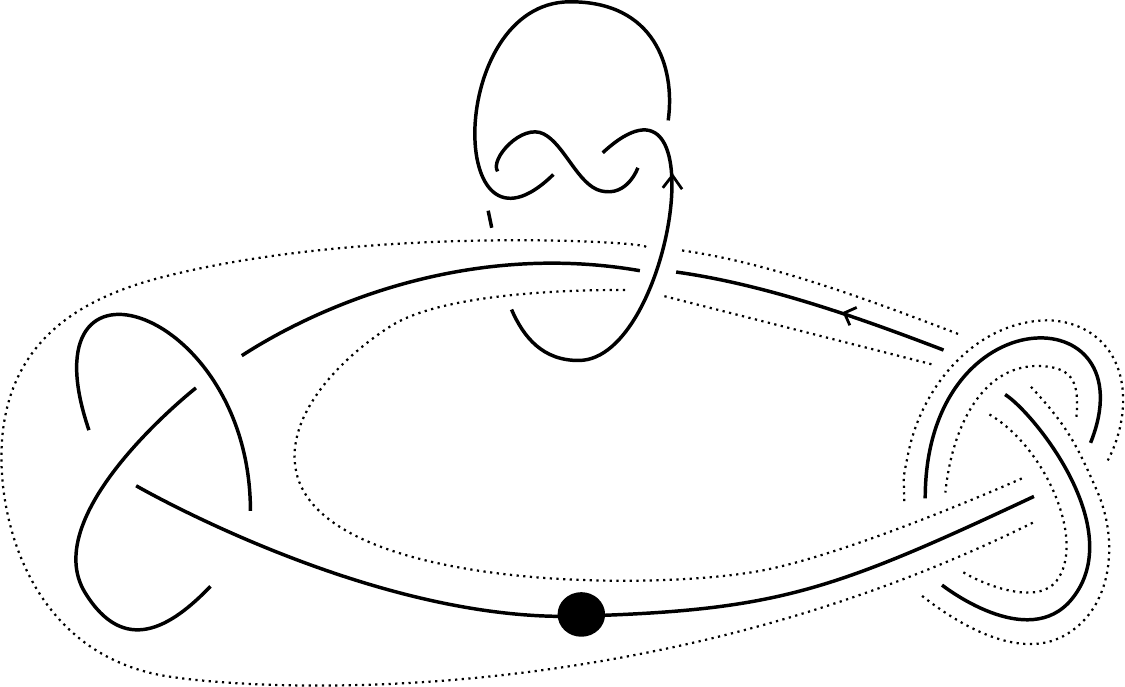}
  \caption{A torus $T$ (shown with dotted lines) embedded in the boundary of a contractible $4$-manifold. }\label{F:toruslooptheorem}
\end{figure}

A straightforward $\pi_1$ computation shows that $W$ is contractible and consequently, $Y$ is a homology sphere. Thus the inclusion of the torus $T$, shown in the figure, extends to a map of a solid torus in~$W$. 

We first show that $T$ is incompressible in $Y$. For the moment, assume that $K$ is a non-trivial knot. We will show that the torus $T'$ in Figure~\ref{F:generaltoruslooptheorem} is incompressible in $Y$ and that $Y$ is irreducible. The torus $T'$ separates $Y$ into two components, $X_1$ and $X_2$, each of which is obtained by performing surgery on a knot in a solid torus (recall that we get a picture for $Y$ by replacing the dot on $L_1$ with a $0$). In this situation,~\cite[Theorem~1.1]{gabai:surgery-in-solid-tori} states that there are three possibilities for each $X_i$, as follows. The first possibility is that $X_i$ is a solid torus. However, we know from~\cite[Lemma~2]{bing-martin:cubes-knotted-holes} that if a knot intersects a generic meridional disk of the solid torus containing it exactly once, the result of surgery is not even a homotopy solid torus. Thus, since our knots satisfy the intersection condition, $X_i$ cannot be a solid torus. The second option is that $X_i$ is a connected sum with a closed manifold with non-trivial $H_1$. This is impossible since $X_i$ is contained in the homology sphere $Y$. This leaves the last possibility, where each $X_i$ is irreducible with incompressible boundary. Thus, $T'$ is incompressible in $Y$ and $Y$ is irreducible. 

Now we return to the torus $T$, which separates $Y$ into two components, one of which contains $L_2$. Let $V$ be the component that does not contain $L_2$. Note that $V$ is once again the result of surgery on a knot in a solid torus, and by the same argument as above, $T$ is incompressible in $V$. Thus, if $T$ is compressible in $Y$, it must be compressible in $Y-V$, i.e.\ there must be an essential curve on $T$ which bounds a disk $\Delta$ in $Y-V$. Since $T'$ is incompressible in $Y$ and $Y$ is irreducible, we can use a standard innermost disk argument (given next) to conclude that $\Delta$ may be chosen to not intersect $T'$. Perform an isotopy to ensure that $\Delta$ and $T'$ are transverse. Thus, the two surfaces intersect in a collection of disjoint circles. Choose an innermost such circle on $\Delta$, i.e.\ a circle $C$ that bounds a disk on $\Delta$. Since $T'$ is incompressible, this circle bounds a disk on $T'$ as well. The union of these two disks form an embedded sphere in $Y$, which bounds a ball since $Y$ is irreducible. Use this ball to isotope $\Delta$ away from $T'$ so that there is one fewer circle of intersection between $\Delta$ and $T'$. (Alternatively, simply throw away the disk bounded by $C$ in $\Delta$ and replace it by the disk bounded by $C$ in $T'$. While this does not require $Y$ to be irreducible, the resulting disk may not be isotopic to the one we started out with.) By induction, we may assume that $\Delta$ and $T'$ are disjoint. Consequently, there must be an essential curve in the boundary of $S^3-J$ which bounds a disk in $S^3-J$. This is impossible since $J$ is a non-trivial knot and thus, $T$ is incompressible in $Y$. Note that this argument cannot possibly work when $K$ is the unknot since then $T'$ is clearly compressible. 

Now assume that $K$ is the unknot and $n\neq 0$. We will show that the torus $T$ is isotopic to the torus $T_0$, shown on the right of Figure~\ref{F:twotori}. This will complete the proof that $T$ is incompressible in $Y$ since we may perform the slam-dunk move~\cite[p 163]{GS99} on $L_2$ to see that $T_0$ is a standard `swallow-follow' torus in the $\frac{-1}{n}$ surgery on $L_1$, which we know is incompressible by~\cite[Lemma~7.1]{gordon:incompressible}.
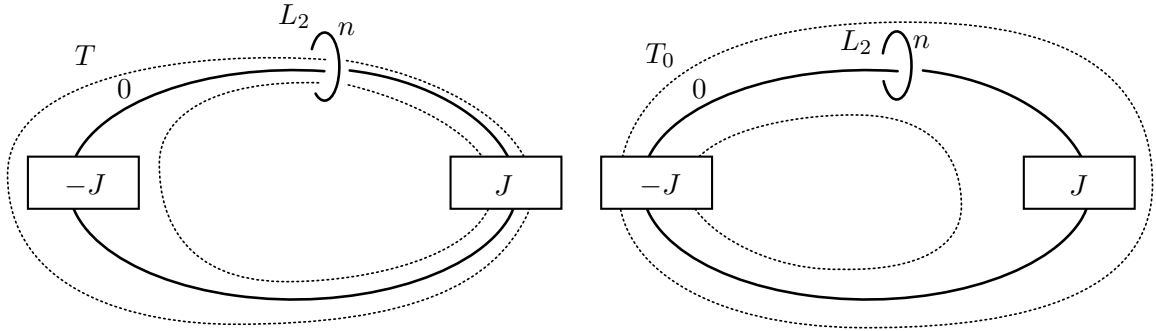
\begin{figure}[htb]
  \centering
    \labellist
    \normalsize\hair 0mm
    \pinlabel {$J$} at 190 55
    \pinlabel {$-J$} at 30 55
		\pinlabel {$J$} at 410 55
    \pinlabel {$-J$} at 250 55
    \pinlabel {$L_2$} at 110 120
    \pinlabel {$n$} at 130 115
		\pinlabel {$L_2$} at 325 110
    \pinlabel {$n$} at 350 110
    \pinlabel {$T$} at 30 105
		\pinlabel {$T_0$} at 250 105
    \pinlabel {$0$} at 45 92
		\pinlabel {$0$} at 265 92
%    \pinlabel {$\alpha$} at 83 107
%    \pinlabel {$\beta$} at 60 80    
      \endlabellist
  \includegraphics[width=\textwidth]{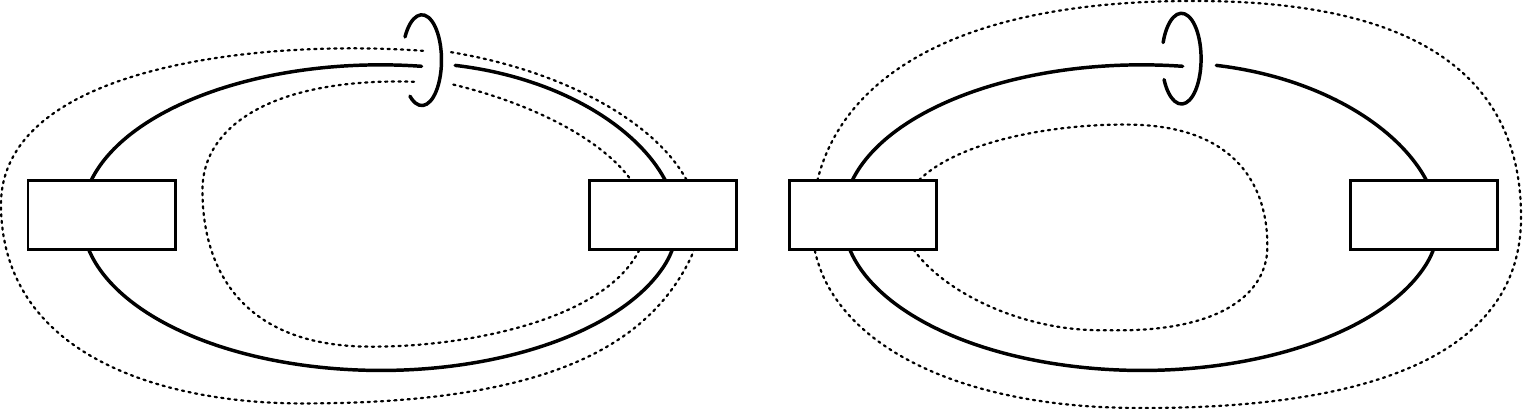}
 \caption{Tori $T$ (left) and $T_0$ (right) in the $3$-manifold $Y$. We claim these are isotopic.}\label{F:twotori}
 \end{figure}

The isotopy from $T$ to $T_0$ follows the standard proof of the slam-dunk move on diagrams for $3$-manifolds, in which $L_2$ is pushed into the surgery torus $N$ for the indicated $0$-surgery on $L_1$. The initial situation, before perfoming $0$-surgery on $L_1$, is shown in Figure~\ref{F:initialT}.  

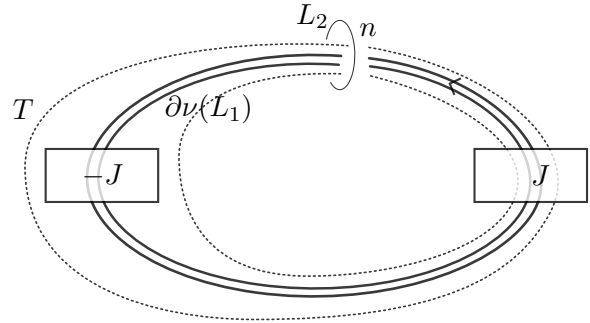
\begin{figure}[h]
  \centering
    \labellist
    \normalsize\hair 0mm
    \pinlabel {$J$} at 195 55
    \pinlabel {$-J$} at 30 55
    \pinlabel {$L_2$} at 109 115
    \pinlabel {$n$} at 130 110
    \pinlabel {$T$} at 0 80
    \pinlabel {$\partial\nu(L_1)$} at 70 80
%\pinlabel {$T_0$} at 38 127
%    \pinlabel {$0$} at 45 95
%    \pinlabel {$\alpha$} at 83 107
%    \pinlabel {$\beta$} at 60 80    
      \endlabellist
  \includegraphics[scale = 1]{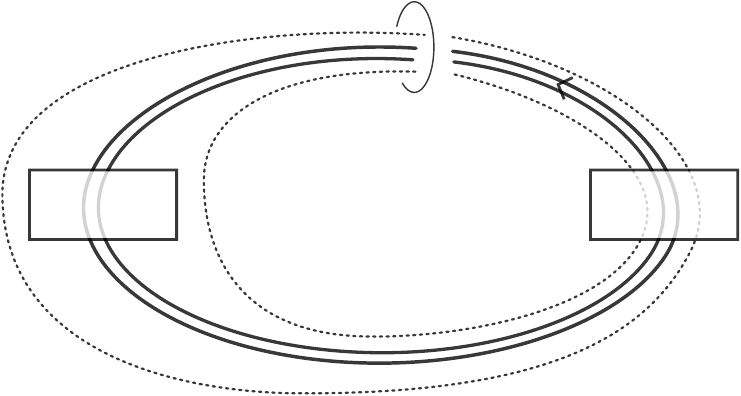}
 \caption{Initial position, showing $L_2$ and $T$ in $Y$ before the isotopy. The dark lines indicate the boundary of a tubular neighborhood of $L_1$.}\label{F:initialT}
\end{figure}
Before we perform $0$-surgery on $L_1$, there is a clear isotopy that pushes $L_2$ and a portion of $T$ into the neighborhood $\nu(L_1)$ indicated in Figure~\ref{F:Kdetail}. Starting from the initial position in Figure~\ref{F:initialT}, perform $0$-surgery on $L_1$. This involves removing $\nu(L_1)$ and gluing in a solid torus $N$ so that $\mu_N \to \lambda_\nu$ and $\lambda_N \to \mu_\nu$ (up to sign). In this surgered manifold, perform an isotopy corresponding to the one discussed above (Figure~\ref{F:Kdetail}), i.e.\ push $L_2$ and a portion of $T$ into $N$. We see that $L_2$ is a core curve of $N$; the portion of $T$ inside $N$ is an annulus $A$ and the portion outside is $A'$. Figure~\ref{F:KandAinN} shows $N$ with $L_2$ and $A$ inside.
 \begin{figure}[h]
 \centering
    \labellist
    \normalsize\hair 0mm
    \pinlabel {$L_2$} at 60 30
     \pinlabel {$A$} at 98 32
    \pinlabel {$T$} at -5 40
    \pinlabel {$\nu(L_1)$} at 15 25   
      \endlabellist
  \includegraphics[scale = 1.3]{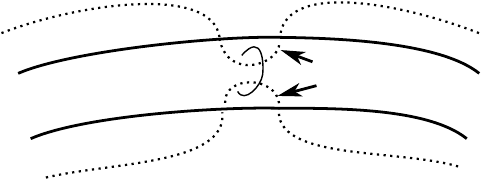}
 \caption{Detail showing $L_2$ pushed into $\nu(L_1)$, with an annular portion $A$ of $T$ pushed in there too.}\label{F:Kdetail}
\end{figure}
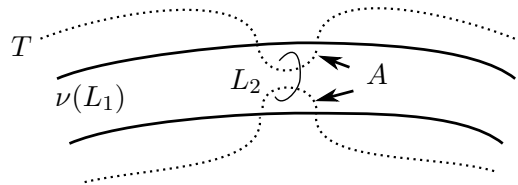
\begin{figure}[htb]
  \centering
    \labellist
    \normalsize\hair 0mm
    \pinlabel {$\partial \nu$} at 30 105
		\pinlabel {$A$} at 52 100
    \pinlabel {$L_2$} at 43 88
      \endlabellist
  \includegraphics[scale = 1.2]{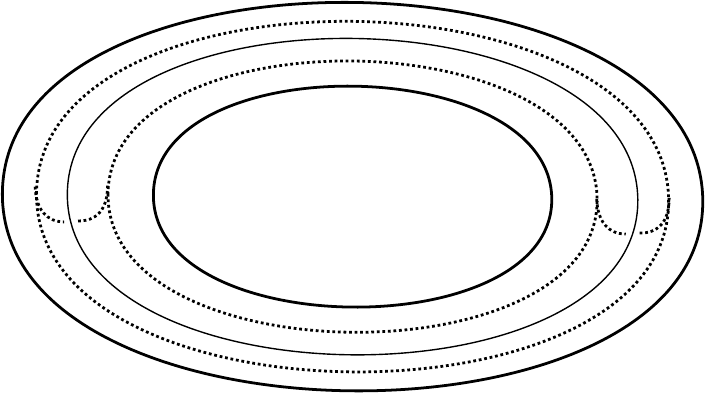}
 \caption{A ``top-down'' view of the solid torus $N$ with the annulus $A$ within. Push the interior of $A$ into the page while keeping the boundary fixed.}\label{F:KandAinN}
 \end{figure}
Within $N$, push the interior of $A$ down to $\partial N$ away from $L_2$, while keeping the boundary fixed. In the original picture (Figure~\ref{F:initialT}), this new annulus is obtained from $\partial\nu$ by removing an annular neighborhood of a meridian, i.e.\ we have isotoped $T$ to the torus shown in Figure~\ref{F:secretlyswallowfollow}, which is clearly seen to be the torus $T_0$ from Figure~\ref{F:twotori}. 
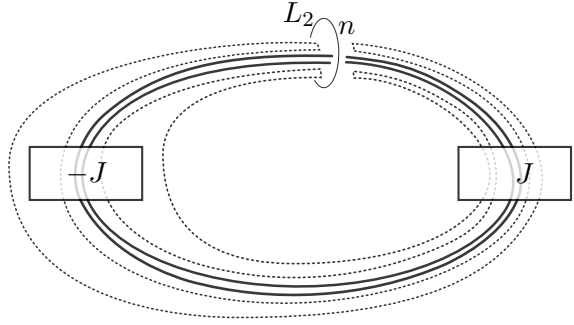
\begin{figure}[h]
  \centering
    \labellist
    \normalsize\hair 0mm
    \pinlabel {$J$} at 195 55
    \pinlabel {$-J$} at 30 55
    \pinlabel {$L_2$} at 110 115
    \pinlabel {$n$} at 128 110 
      \endlabellist
  \includegraphics[scale = 1]{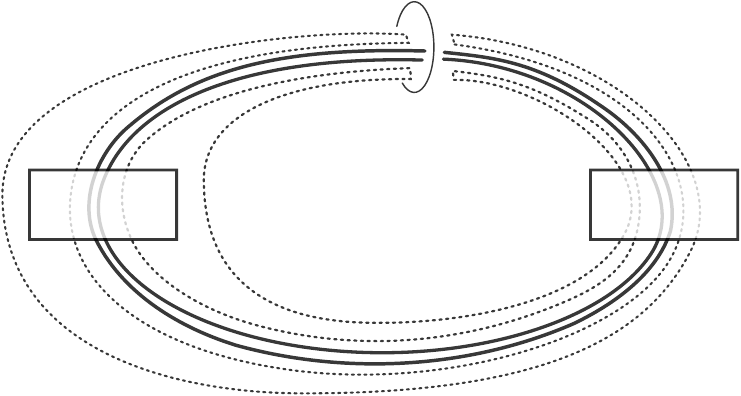}
 \caption{The result of performing an isotopy on $T$.}\label{F:secretlyswallowfollow}
\end{figure}

Thus, the torus $T$ is incompressible in $Y$, except for the single case when $K$ is the unknot and $n=0$, where it is easy to see that $T$ is compressible.

If $T$ bounds a properly embedded solid torus in $W$, there must be some essential curve on $T$, denoted $\gamma$, with zero linking with its pushoff on $T$, that bounds a properly embedded disk in $W$, i.e.\ is slice in $W$. Our first step is thus to determine which essential curves on $T$ have zero self-linking in $Y$. Note that $Y$ is obtained from $S^3$ by performing surgery on the link $\{L_1,L_2\}$. The linking matrix for $\{L_1,L_2\}$ is $ B=\left[ \begin{array}{cc}
0 & 1 \\
1 & n \end{array} \right]$. Let $\lk_{S^3}(\cdot, \cdot)$ and $\lk_{Y}(\cdot, \cdot)$ denote the linking number of curves in $S^3$ and $Y$ respectively. We use Hoste's formula from~\cite{Hoste86} for the linking number of homologically trivial curves in a $3$-manifold $Y$ given by a surgery diagram, which states that 
$$\lk_Y(\sigma,\eta)=\lk_{S^3}(\sigma, \eta) - \left[ \begin{array}{cc} \lk_{S^3}(\sigma,L_1) & \lk_{S^3}(\sigma,L_2) \end{array} \right] B^{-1} \left[ \begin{array}{cc} \lk_{S^3}(\eta,L_1) & \lk_{S^3}(\eta,L_2) \end{array} \right]^T$$
for any two curves $\sigma$ and $\eta$ in $Y$. We compute that $\lk_{S^3}(\alpha, L_1)=1$, $\lk_{S^3}(\alpha, L_2)=0$, $\lk_{S^3}(\beta, L_1)=0$, and $\lk_{S^3}(\beta, L_2)=1$. Additionally, the linking of $\alpha$ with its tangential pushoff, $\lk_{S^3}(\alpha, \alpha^+)=0$, and the linking of $\beta$ with its tangential pushoff, $\lk_{S^3}(\beta, \beta^+)=0$. Moreover, $\lk_{S^3}(\alpha, \beta^+)=0$ and $\lk_{S^3}(\beta, \alpha^+)=1$. Using this information, it is straightforward to apply the above formula to see that $x,y \in\Z$, the linking in $Y$ of a curve $\gamma$ in the homology class $x[\alpha]+y[\beta]$ with its tangential pushoff is given by $nx^2-xy$. Thus, the homology class of $\gamma$ is (up to an irrelevant sign) either $[\beta]$ or $[\alpha]+n[\beta]$. Since $T$ is a torus, there is a unique free homotopy class of curves associated with each primitive homology class. In the remainider of the proof, we will choose representatives of $[\beta]$ and $[\alpha]+n[\beta]$ and show that neither is slice in $W$.

\begin{figure}[htb]
  \centering
    \labellist
    \normalsize\hair 0mm
    \pinlabel {$J$} at 190 55
    \pinlabel {$-J$} at 25 55
    \pinlabel {$K$} at 137 130
    \pinlabel {$n$} at 132 155
    \pinlabel {$\beta'$} at 82 60    
    \pinlabel {$L_1$} at 30 25
    \pinlabel {$L_2$} at 102 155    
      \endlabellist
  \includegraphics[scale = 1]{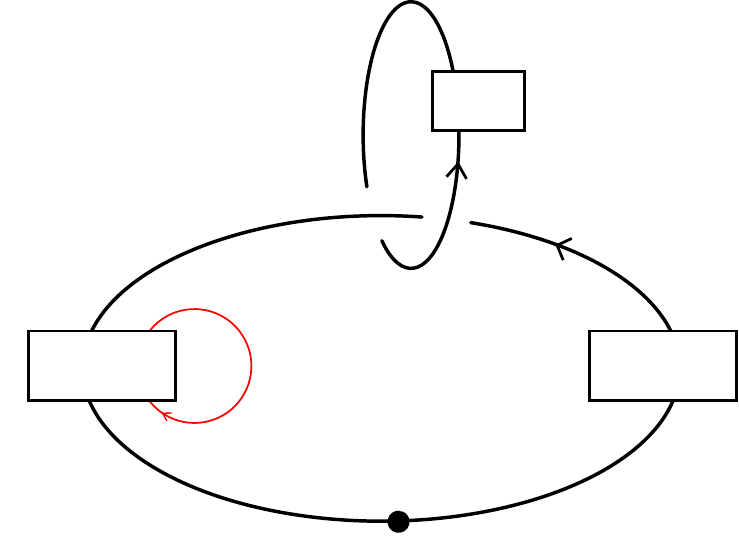}
  \caption{The curve $\beta'$ is isotopic in $Y$ to $\beta$ from Figure~\ref{F:generaltoruslooptheorem}.}\label{F:betaprime}
\end{figure}
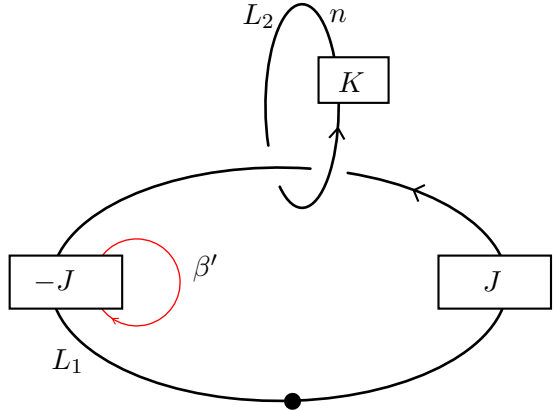

Let $\gamma =\beta$. By sliding $\beta$ under (the reverse of) $L_1$, we see that $\beta$ is isotopic in $Y$ to a curve $\beta'$ that has zero linking number with both $L_1$ and $L_2$ (see Figure~\ref{F:betaprime}) and moreover, we can see a Seifert surface for this curve, namely we can use one for $-J$. The curves on this Seifert surface have zero linking with $L_1$ and $L_2$, and thus, by Hostes's formula, their linking numbers in $Y$ are the same as in $S^3$. Then, as long as we choose $J$ such that $-J$ has, say, non-zero signature as a knot in $S^3$, $\gamma$ is not slice in $W$. Here, all we use is that the signature can be computed from the Seifert matrix (from above, this will be the same as a Seifert matrix for $J$ in $S^3$) and obstructs sliceness in a homology ball. 

\begin{figure}[htb]
  \centering
    \labellist
    \normalsize\hair 0mm
    \pinlabel {$J$} at 190 55
    \pinlabel {$-J$} at 25 55
    \pinlabel {$K$} at 137 130
    \pinlabel {$n$} at 132 155
    \pinlabel {$\gamma'$} at 172 130    
    \pinlabel {$L_1$} at 30 25
    \pinlabel {$L_2$} at 102 155    
      \endlabellist
  \includegraphics[scale = 1]{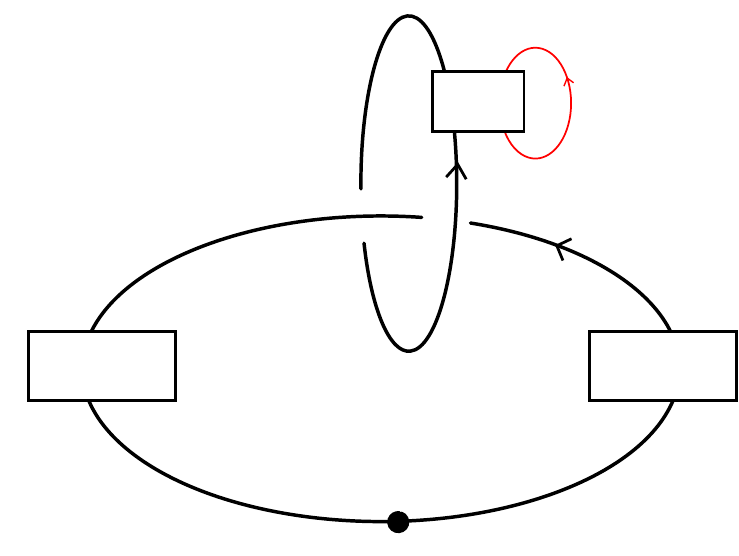}
  \caption{The curve $\alpha$ from Figure~\ref{F:generaltoruslooptheorem} is isotopic in $Y$ to $\gamma'$.}\label{F:alphaprime}
\end{figure}
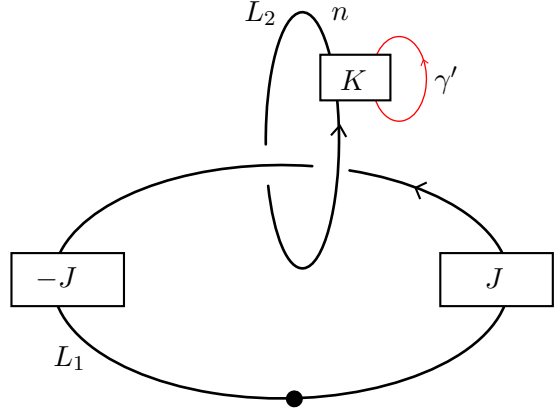
\begin{figure}[htb]
  \centering
    \labellist
    \normalsize\hair 0mm
    \pinlabel {$J$} at 185 55
    \pinlabel {$-J$} at 20 55
    \pinlabel {$K$} at 130 132
    \pinlabel {$3$} at 125 155
    \pinlabel {$\gamma$} at 150 65    
    \pinlabel {$L_1$} at 20 25
    \pinlabel {$L_2$} at 90 155    
		\pinlabel {$J$} at 430 55
    \pinlabel {$-J$} at 265 55
    \pinlabel {$K$} at 375 132
    \pinlabel {$-3$} at 370 155
    \pinlabel {$\gamma$} at 395 70    
    \pinlabel {$L_1$} at 270 25
    \pinlabel {$L_2$} at 340 155  
		\pinlabel {$n>0$} at 100 -5
		\pinlabel {$n<0$} at 345 -5
      \endlabellist
  \includegraphics[scale = 0.8]{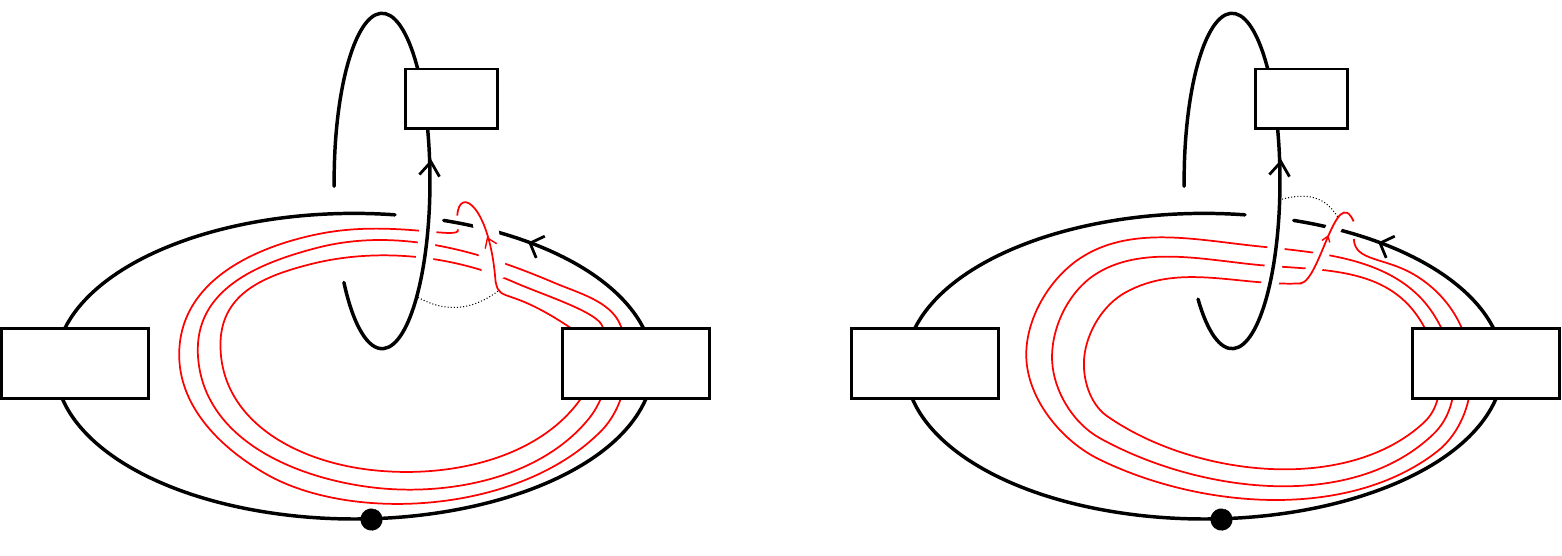}
  \caption{The specific cases of $n=3$ and $n=-3$ are shown. For other values of $n$, $\gamma$ would wind around the box labeled $J$ and through $L_2$ $n$ times. }\label{F:gamma}
\end{figure}

\begin{figure}[htb]
  \centering
    \labellist
    \normalsize\hair 0mm
    \pinlabel {$J$} at 185 55
    \pinlabel {$-J$} at 20 55
    \pinlabel {$K$} at 130 132
    \pinlabel {$3$} at 127 112		
    \pinlabel {$3$} at 120 157
    \pinlabel {$\gamma'$} at 150 63    
    \pinlabel {$L_1$} at 20 25
    \pinlabel {$L_2$} at 90 155    
		\pinlabel {$J$} at 430 55
    \pinlabel {$-J$} at 265 55
    \pinlabel {$K$} at 375 132
		\pinlabel {$-3$} at 372 111	
    \pinlabel {$-3$} at 367 157
    \pinlabel {$\gamma'$} at 395 68    
    \pinlabel {$L_1$} at 270 25
    \pinlabel {$L_2$} at 340 155  
		\pinlabel {$n>0$} at 100 -5
		\pinlabel {$n<0$} at 345 -5
      \endlabellist
  \includegraphics[scale = 0.8]{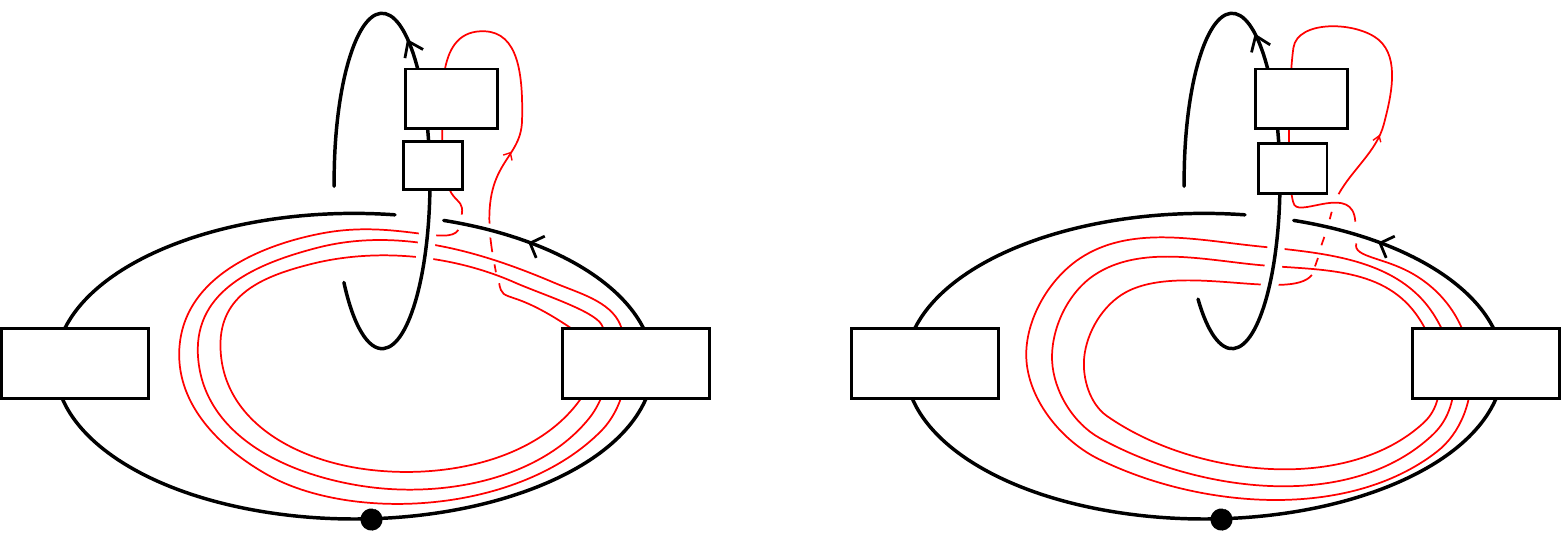}
  \caption{The specific cases of $n=3$ and $n=-3$ are shown. The box labeled with $3$ (or $n>0$ in the general case) indicates 3 (or $n$) right-handed twists. The box labeled with $-3$ (or $n<0$ in the general case) indicates $3$ (or $n$) left-handed twists.}\label{F:gammaprime}
\end{figure}
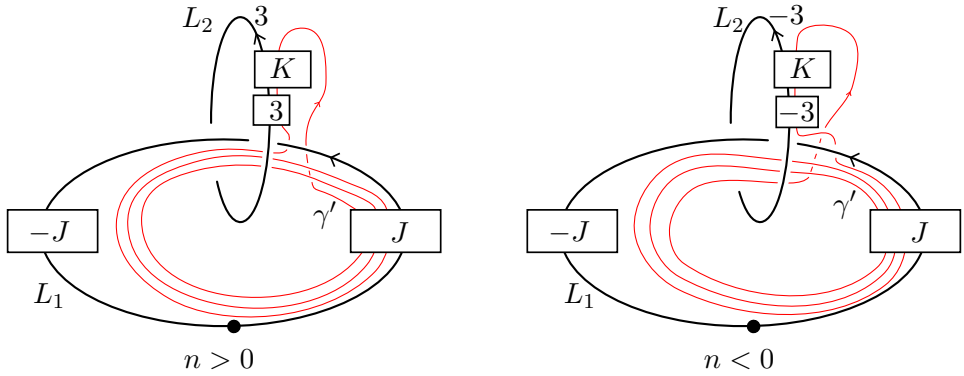

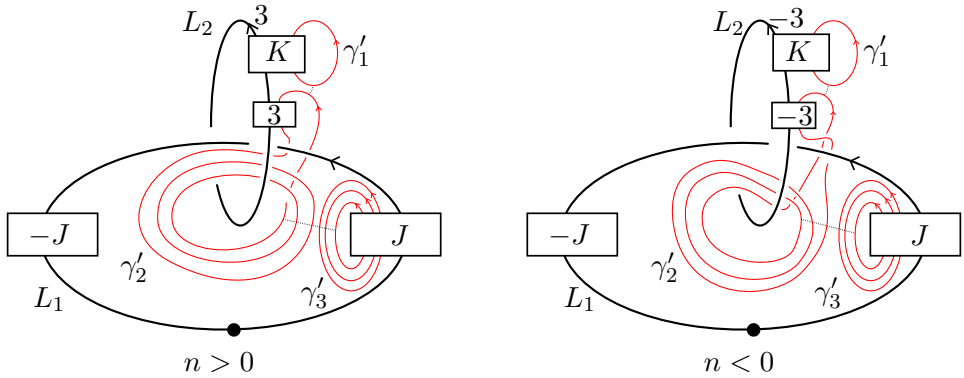
\begin{figure}[htb]
  \centering
    \labellist
    \normalsize\hair 0mm
    \pinlabel {$J$} at 185 55
    \pinlabel {$-J$} at 20 55
    \pinlabel {$K$} at 128 142
    \pinlabel {$3$} at 126 112		
    \pinlabel {$3$} at 120 159
    \pinlabel {$\gamma'_1$} at 165 142  
    \pinlabel {$\gamma'_2$} at 60 40 
		\pinlabel {$\gamma'_3$} at 145 28 
    \pinlabel {$L_1$} at 20 25
    \pinlabel {$L_2$} at 90 155    
		\pinlabel {$J$} at 430 55
    \pinlabel {$-J$} at 265 55
    \pinlabel {$K$} at 373 142
		\pinlabel {$-3$} at 371 111	
    \pinlabel {$-3$} at 367 157   
    \pinlabel {$L_1$} at 270 25
    \pinlabel {$L_2$} at 340 155  
		\pinlabel {$n>0$} at 100 -5
		\pinlabel {$n<0$} at 345 -5
    \pinlabel {$\gamma'_1$} at 410 142  
    \pinlabel {$\gamma'_2$} at 310 40 
		\pinlabel {$\gamma'_3$} at 387 28 
      \endlabellist
  \includegraphics[scale = 0.8]{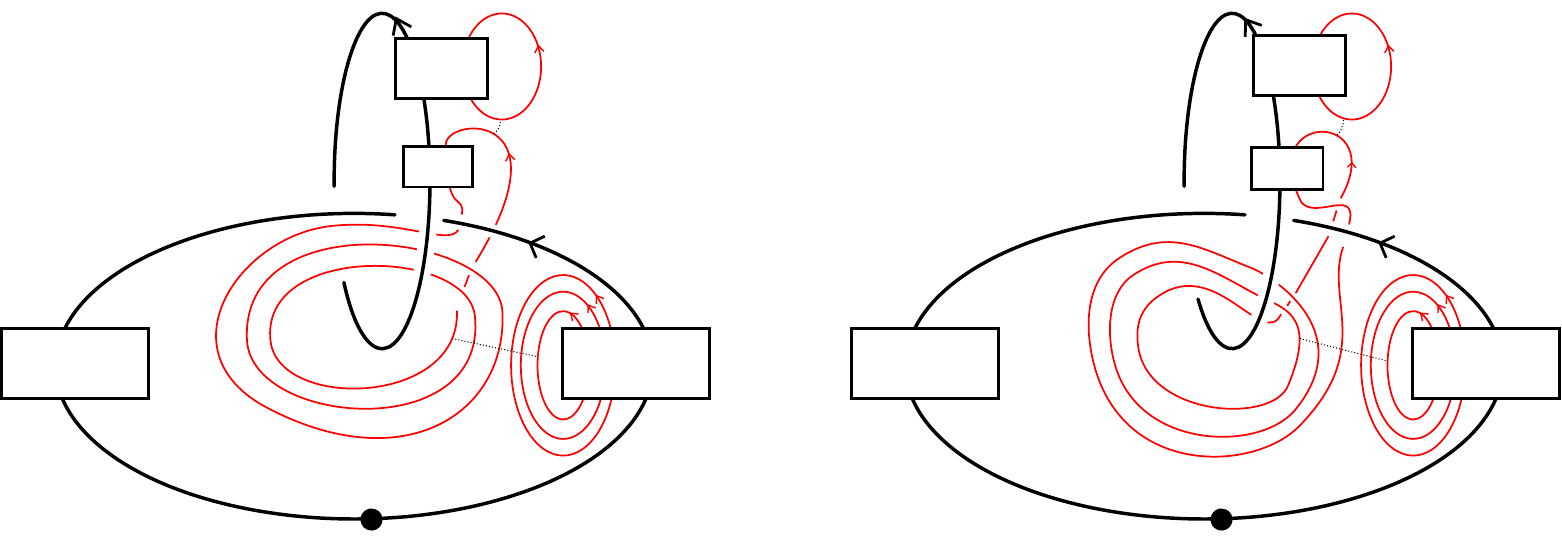}
  \caption{The specific cases of $n=3$ and $n=-3$ are shown. The box labeled with $3$ (or $n>0$ in the general case) indicates $3$ (or $n$) right-handed twists.}\label{F:gammaprimepieces}
\end{figure}

Now suppose $\gamma$ is a curve with homology class $[\alpha]+n[\beta]$. When $n=0$, this class is represented by the curve $\alpha$. By sliding $\alpha$ over (the reverse of) $L_2$, we see that $\alpha$ is isotopic in $Y$ to a curve $\gamma'$ indicated in Figure~\ref{F:alphaprime}. As in the previous case, we see a Seifert surface for $\gamma'$ given by a parallel pushoff of one for $K$, such that the curves on the surface do not link with $L$. As a result, if we choose $K$ to be a knot with say, non-zero signature as a knot in $S^3$, $\gamma$ is not slice in $W$. The $n\neq 0$ case is similar, but takes a bit more work. Depending on whether $n$ is positive or negative, we let $\gamma$ be the curve shown in Figure~\ref{F:gamma}, where it is easy to see that $\gamma$ has homology class $[\alpha]+n[\beta]$. 
Slide $\gamma$ over (the reverse of) $L_2$ along the dotted paths shown in the figure to obtain the curve $\gamma'$ shown in Figure~\ref{F:gammaprime}. A Seifert surface for $\gamma'$ can be constructed using $n$ parallel (disjoint) copies of a Seifert surface for $J$ and a copy of a Seifert surface for $K$. We show this in Figure~\ref{F:gammaprimepieces} and describe the procedure next. Here we draw a curve $\gamma'_1$, an unknot $\gamma'_2$ which is unlinked from all other curves, and an $n$--component link $\gamma'_3$. The components of $\gamma'_3$ bound a collection of $n$ disjoint parallel copies of a Seifert surface for $J$. The curve $\gamma'_2$ bounds a disk in $Y$ away from all other curves and surfaces. The curve $\gamma'_1$ bounds a parallel copy of a Seifert surface for $K$. The curve $\gamma'$ can be obtained from $\gamma'_1$, $\gamma'_2$ and $\gamma'_3$ by band summing along the dotted lines shown in the figure (for $\gamma'_2$ and $\gamma'_3$ use a collection of $n$ nested bands); we can use the same bands to fuse the various surfaces together to get a Seifert surface for $\gamma'$. Note that the curves on this surface do not interact with either $L_1$ or $L_2$ and therefore, in algebraic concordance, $\gamma'$ corresponds to the class of $K\conn J_{n,1}$ in $S^3$, where $J_{n,1}$ is the $(n,1)$--cable of $J$. 

From the above argument, we see that if $K$ and $J$ are chosen such that neither $-J$ nor $K\#J_{n,1}$ is algebraically slice, then $T$ does not extend to an embedded solid torus. For example, we could choose $K$ and $J$ to both be left- or right-handed trefoils. \end{proof}

The above result shows that the failure of certain `surgery curves' $\gamma$ lying on $T$ to be slice in $W$ gives rise to obstructions to $T$ bounding a solid torus in $W$. We now give a converse result, showing that under certain reasonable conditions, the sliceness of such curves is the complete obstruction.  

Recall that $Y$ denotes the $3$-manifold boundary of a compact $4$-manifold $W$, and $T\subseteq Y$ is an embedded torus. 

\begin{proposition}\label{P:embed} Let $T$ be separating and $\gamma \subset T$ be an embedded circle. Let $e$ be the framing of $\gamma$ in $Y$ determined by the pushoff of $\gamma$ in $T$. If 
\begin{enumerate}
\item $[\gamma] \neq 0 \in H_1(T)$,
\item $\gamma$ is smoothly (resp. topologically) slice in $W$ with respect to $e$ (i.e.\ $\gamma$ and its pushoff in $T$ bound disjoint slice disks in $W$), and
\item the surgered manifold $Y_e(\gamma)$ is irreducible,
\end{enumerate}
then $T$ bounds a smoothly (resp. topologically) embedded solid torus in $W$.
\end{proposition}
The proof is identical in the smooth and topological categories, so we will omit those adjectives.  The difference in the outcome of the proof is due entirely to the smoothness of the embedded $2$-handle described in the proof. 
\begin{proof}
We view $S^1 \times D^2$ as having a handlebody decomposition relative to $T=\partial (S^1\times D^2$, with a single $2$-handle attached along $\gamma$, followed by a $3$-handle. The goal is to perform the necessary handle attachments ambiently in $W$. By assumption $(2)$, we can add a $4$-dimensional $2$-handle to $Y$ along $\gamma$, such that the trace of the surgery is embedded in $W$ and that this $2$-handle contains a $3$-dimensional $2$-handle added to $T$.  Since $[\gamma] \neq 0$, the torus $T$ has been surgered to yield a $2$-sphere $S \subset Y_e(\gamma)$.  By assumption $(3)$, the sphere $S$ bounds a ball in $Y_e(\gamma)$, and this ball then serves as the remaining $3$-handle in $S^1 \times D^2$.
\end{proof}

Lastly, we address the difference between the smooth and topological categories in the following result. 

\begin{theorem}\label{T:torus-smooth} There exists a contractible $W$ and a torus $T\subseteq Y$, where $T$ extends to a topological embedding of a solid torus in $W$, but not a smooth embedding. \end{theorem}

\begin{proof}Consider the manifold $W$ and torus $T$ used in Theorem~\ref{T:nodehnfortori} (see Figure~\ref{F:generaltoruslooptheorem}), with $K$ the positive Whitehead double of the right-handed trefoil, $J$ the left-handed trefoil, and $n=0$. We saw earlier that the curve $\alpha$ shown in Figure~\ref{F:generaltoruslooptheorem} is non-trivial in $H_1(T)$ and has zero self-linking in $Y$. In order to apply Proposition~\ref{P:embed}, we need to show that the manifold in Figure~\ref{F:irreducible} is irreducible and that $\alpha$ (with the zero framing) is slice in $W$. For the first, we perform a slam dunk move to see that our manifold in Figure~\ref{F:irreducible} is the 0--surgery on $K$, which is well-known to be irreducible~\cite[Corollary 5]{gabai:zerosurgery}. From the Seifert surface for $\alpha$ constructed in the proof of Theorem~\ref{T:nodehnfortori} we know that $\alpha$ has Alexander polynomial one, and thus is slice in $W$ by~\cite[Theorem 11.7B]{freedman-quinn}. Thus, $T$ bounds a topologically embedded solid torus in $W$. 

\begin{figure}[htb]
  \centering
    \labellist
    \normalsize\hair 0mm
    \pinlabel {$K$} at 188 140
    \pinlabel {$0$} at 130 101
    \pinlabel {$0$} at 162 152	
    \pinlabel {$0$} at 90 25
      \endlabellist
  \includegraphics[scale =0.8]{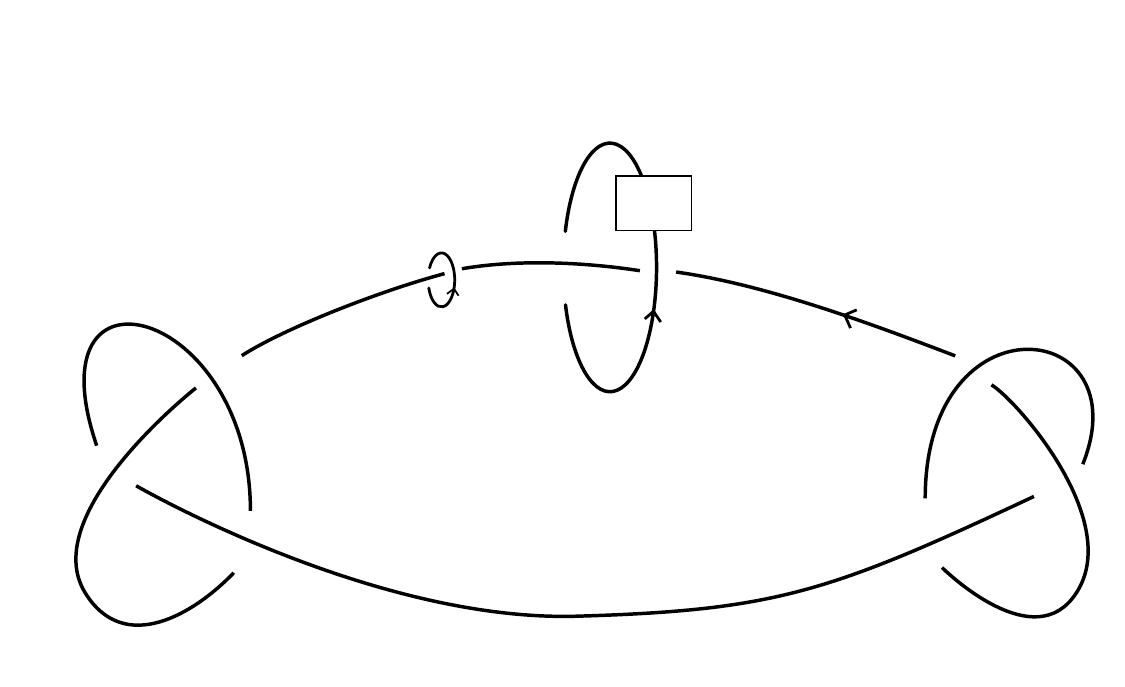}
  \caption{Here $K$ is the positive Whitehead double of the right-handed trefoil. }\label{F:irreducible}
\end{figure}

Next we show that $T$ cannot be the boundary of a smoothly embedded solid torus. As in Theorem~\ref{T:nodehnfortori}, using linking number considerations, we know that if such a smoothly embedded solid torus could be found, either $\alpha$ or $\beta$ (pictured in Figure~\ref{F:generaltoruslooptheorem}) would have to be smoothly slice. Recall that $\beta$ is isotopic to the curve $\beta'$ shown in Figure~\ref{F:betaprime}, and as we saw earlier, since $-J$ has non-zero signature (as $J$ is the left-handed trefoil), $\beta$ is not smoothly slice in $W$. Thus, it remains to show that $\alpha$ is not slice in $W$. The following is inspired by~\cite[Figures 4 and 6]{akbulut:infinite}. We start with a diagram for $W$ in our specific situation in Figure~\ref{F:nosmoothsolidtorus1} where $\alpha$ is drawn in red; we redraw $W$ using unknotted circles, decorated with dots, in Figure~\ref{F:nosmoothsolidtorus2}. 

\begin{figure}[hbt]
  \centering
    \labellist
    \normalsize\hair 0mm
    \pinlabel {$K$} at 188 140
    \pinlabel {$\alpha$} at 130 101
    \pinlabel {$0$} at 162 152	
    \pinlabel {$L_1$} at 90 25
    \pinlabel {$L_2$} at 102 125 
      \endlabellist
  \includegraphics[scale =0.8]{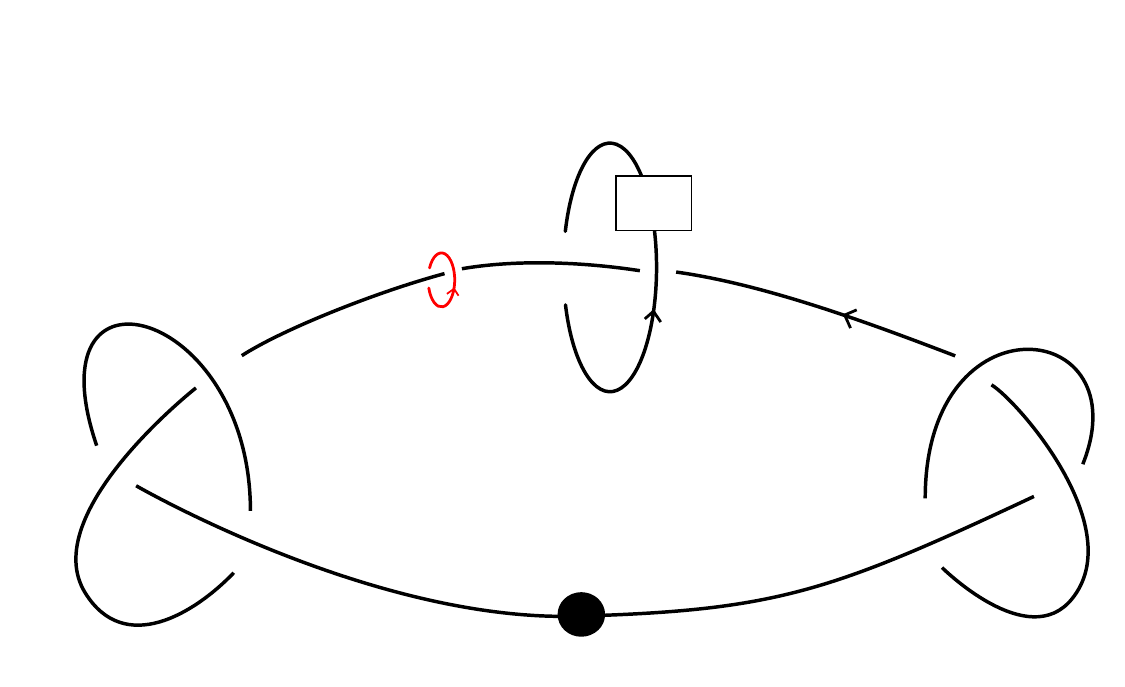}
  \caption{Here $K$ is the positive Whitehead double of the right-handed trefoil. }\label{F:nosmoothsolidtorus1}
\end{figure}
\begin{figure}[htb]
  \centering
    \labellist
    \normalsize\hair 0mm
    \pinlabel {$K$} at 299 102
    \pinlabel {$\alpha$} at 262 82
		\pinlabel {$0$} at 220 25
		\pinlabel {$0$} at 285 80		
      \endlabellist
  \includegraphics[scale = 1]{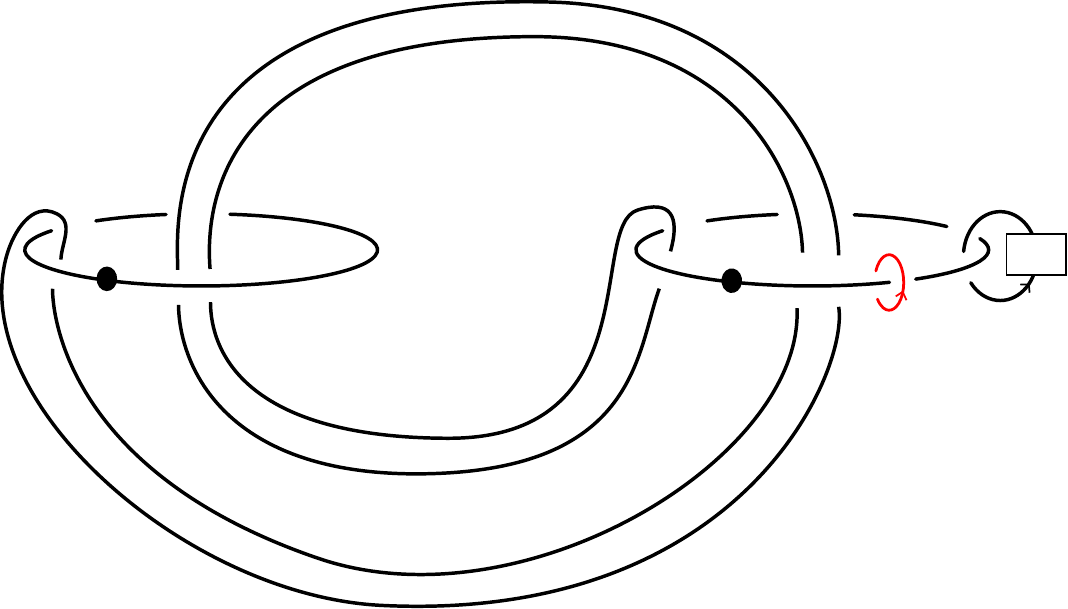}
  \caption{Here $K$ is the positive Whitehead double of the right-handed trefoil knot. }\label{F:nosmoothsolidtorus2}
\end{figure}

Next we switch to denoting $1$-handles using $2$-spheres in Figure~\ref{F:nosmoothsolidtorus3}. We rotate the bottom left $2$-sphere by $360^\circ$ along a horizontal axis, as described in~\cite[Figure 5.42]{GS99}, and perform an isotopy of $L_1$ to get Figure~\ref{F:nosmoothsolidtorus4} where the attaching circles of the $2$-handles are shown in Legendrian position. 
\begin{figure}[htb]
  \centering
    \labellist
    \normalsize\hair 0mm
    \pinlabel {$K$} at 90 140
    \pinlabel {$\alpha$} at 150 160
		\pinlabel {$0$} at 50 75
		\pinlabel {$0$} at 120 132		
      \endlabellist
  \includegraphics[scale=1]{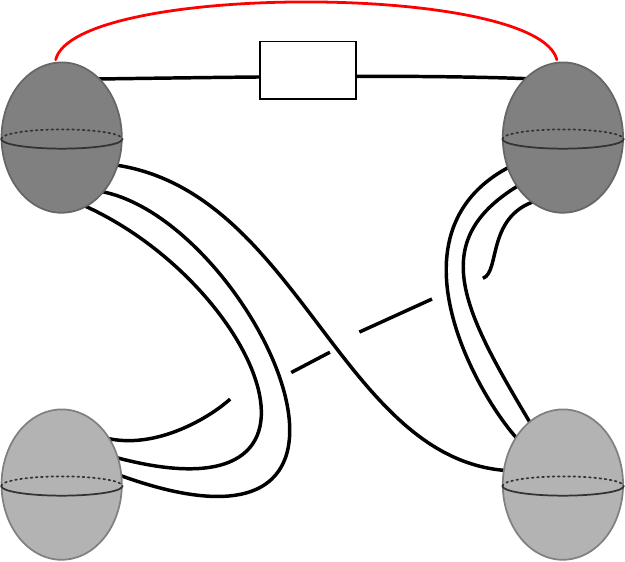}
  \caption{The 2--spheres are identified in pairs horizontally as indicated, by a reflection across a vertical plane separating each pair.}\label{F:nosmoothsolidtorus3}
\end{figure}

\begin{figure}[htb]
  \centering
    \labellist
    \normalsize\hair 0mm
    \pinlabel {$\alpha$} at 80 135
		\pinlabel {$\mathcal{L}_1$} at 95 60
		\pinlabel {$\mathcal{L}_2$} at 120 250			
		\pinlabel {$-1$} at 42 77
		\pinlabel {$0$} at 100 155	
		\endlabellist
  \includegraphics[scale=1]{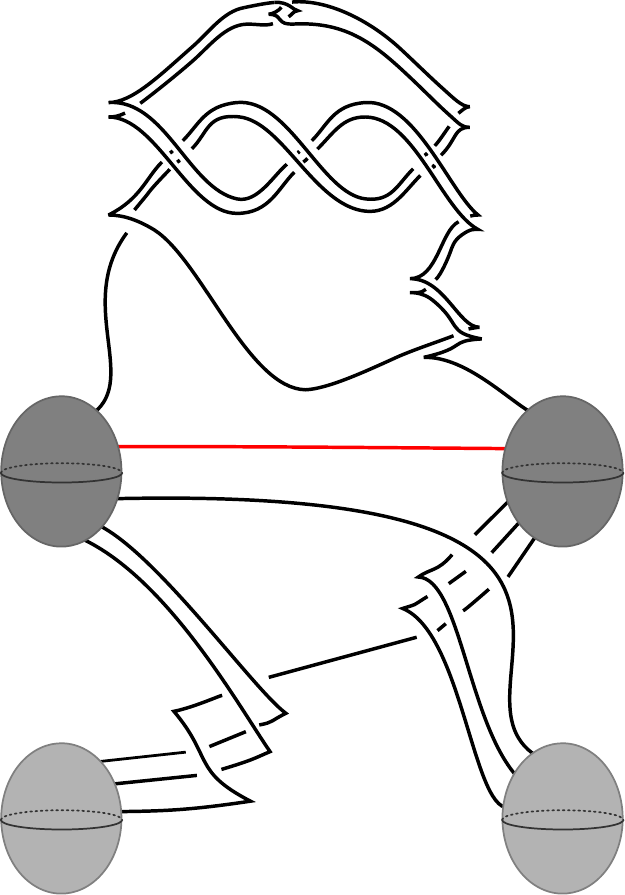}
  \caption{}\label{F:nosmoothsolidtorus4}
\end{figure}

We compute $\tb(\mathcal{L}_1)=0$, and thus the framing of $\mathcal{L}_1$, which is $-1$, is equal to $\tb(\mathcal{L}_1)-1$. Similarly, we compute $\tb(\mathcal{L}_2)=1$, and thus the framing of $\mathcal{L}_2$, which is $0$, is $\tb(\mathcal{L}_2)-1$. Thus, the manifold $W$ admits a Stein structure since it satisfies the condition given in~\cite{Eliashberg90,Gompf98}. We compute from our diagram that $\tb(\alpha)=\rot(\alpha)=0$. Since $W$ is Stein, the slice genus of $\alpha$ in $W$, denoted $g_4^W(\alpha)$, must satisfy the slice--Bennequin inequality of~\cite{AM97}. In particular, $0\leq 2g_4^W(\alpha)-1$, i.e.\ $1\leq g_4^W(\alpha)$ as needed. \end{proof}

%%%%%%%%%%%%%%%%%%%%%%%%%%%%%%%%%%%%%%%%%%%%%%%%%%%%
\section{Extending diffeomorphisms and extending embedded tori}\label{S:diff}
One of our original motivations for investigating tori in the boundary of $4$-manifolds was a relationship with the question of extending certain diffeomorphisms of the boundary over the $4$-manifold.  Suppose one is given a (co-oriented) torus $T \subset Y$, and an integral homology class $\ba \in H_1(T)$. One may define a diffeomorphism $f_{\ba}: Y \to Y$ that is the identity outside of a neighborhood $T \times [0,1]$ as follows. Regard $T$ as $\R^2/\Z^2$, with $[z]$ denoting the equivalence class of the vector $z$. Then we view $\ba$ as a vector in $\Z^2$, and write $f_{\ba}$ on the product neighborhood as
\begin{equation}\label{E:twist}
f_{\ba}([z],t) = ([z+t\ba ],t).
\end{equation}
This diffeomorphism is usually called a Dehn twist on $Y$ with slope $\ba$ along $T$.

The connection to our work is the observation that if $Y = \partial W$ so that $T$ bounds a smooth solid torus in $W$, then $f_{\ba}$ extends to a diffeomorphism of $W$.  A formula for the extension, which might be called a Dehn twist on $W$ along the solid torus, is given in~\cite[\S 3]{akbulut-ruberman:absolute}. Gompf~\cite{gompf:infinite-cork,gompf:infinite-cork-handlebody} has recently constructed an infinite order cork, i.e.\ a contractible $4$-manifold $C$ with a diffeomorphism $f: \partial C \to \partial C$ such that no power of $f$ extends to a diffeomorphism of $C$. Indeed, cutting out $C$ from a particular $4$-manifold, and re-gluing via powers $f^k$ produces distinct $4$-manifolds. According to general principles from $3$-manifold topology, $\partial C$ contains an incompressible torus $T$, and by construction, $f$ is a Dehn twist along this torus. (The idea of looking at such twists in this context was introduced in~\cite{akbulut:infinite}.)  It follows directly that $T$ does not bound a solid torus in $C$. Gompf's family of infinite corks are of the form given in Figure~\ref{F:gompfcork} (see~\cite[Figure~5]{gompf:infinite-cork-handlebody}), with $J$ in an infinite family of double twist knots and $n\neq 0$, and the function $f$ is a Dehn twist along the longitude of the torus $T_0$. 
%In~\cite{gompf:infinite-cork}, he asks whether the manifold pictured is a cork for all other choices of the knot $J$ with respect to a Dehn twist along the torus $T_0$. We will show in a moment that the answer is no. 

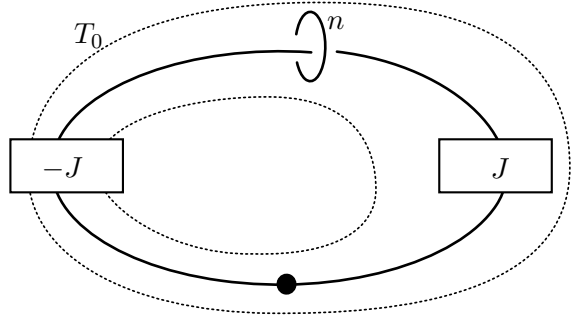
\begin{figure}[htb]
  \centering
    \labellist
    \normalsize\hair 0mm
    \pinlabel {$J$} at 185 55
    \pinlabel {$-J$} at 20 55
    \pinlabel {$n$} at 123 110
    \pinlabel {$T_0$} at 30 105   
      \endlabellist
  \includegraphics{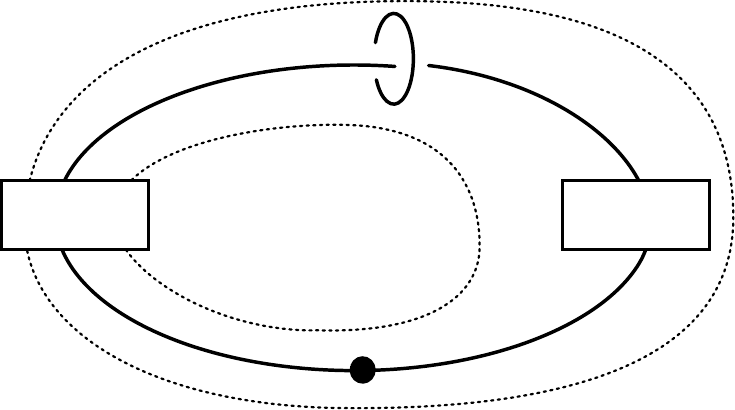}
 \caption{Gompf's infinite corks.}\label{F:gompfcork}
 \end{figure}

In the proof of Theorem~\ref{T:nodehnfortori}, we saw that $T_0$ is isotopic, in the $3$-manifold boundary, to the torus $T$ shown in Figure~\ref{F:generaltoruslooptheorem} and thus, the methods of Section~\ref{S:tori} give a different proof that, for many choices of $J$, the torus $T_0$ does not bound an embedded solid torus; for example, by the proof of Theorem~\ref{T:nodehnfortori}, this holds if neither of $J$ and $J_{n,1}$ is algebraically slice. Gompf's family of double twist knots contains infinitely many twist knots with this property. 

One might wonder whether there is a sort of converse to the observation above, namely, if $T \subset Y = \partial W$ and $f_\ba$ extends over $W$ for all $\ba \in H_1(T)$, does $T$ bound a solid torus in $W$? Combining our results with a construction from~\cite{gompf:infinite-cork-handlebody}, we show that this converse does not hold. First we need a standard observation.

\begin{lemma}\label{L:seifert}
Let $N$ be a $3$-manifold with boundary $T$ with an $S^1$ action that preserves $T$ and restricts to a free action on $T$. Let $\ba$ be the homology class of an orbit on $T$. Then the Dehn twist $f_\ba$ is isotopic to $\id_N$ rel boundary, i.e. there is an isotopy to $\id_N$ that is the identity on $T$. 
\end{lemma}
\begin{proof}
Regard $S^1$ as $\R/\Z$, and write the action of $t\in S^1$ on a point $x$ of $N$ as $t\cdot x$. On the boundary, since the action is free and $\ba$ is an orbit, this can be written as $t \cdot [z] = [z + t\ba]$ as in the definition of $f_\ba$ (here, as before, we regard $T$ as $\R^2/\Z^2$, with $[z]$ denoting the equivalence class of the vector $z$). Parameterize a neighborhood of the boundary as $T \times [0,1]$, where the boundary corresponds to $T \times \{0\}$; by an isotopy, we may assume that the circle action preserves the second factor.  Choose a small $\epsilon$ and a smooth function $\beta(s): [0,1] \to [0,1]$ such that 
$$
\beta(s) = 0\ \text{for}\ s <\epsilon\ \text{and} \ \beta(s) = 1\ \text{for}\ s > 1-\epsilon.
$$
Then an isotopy $F_{\ba,s}$, $s \in [0,1]$ is defined on $T \times [0,1]$ by the formula
$$
F_{\ba,s}([z],t) = ([z + \beta(s) t\ba],t) 
$$
and on $N - T \times [0,1]$ by $F_{\ba,s}(x) = \beta(s) \cdot x$.  For $t=0$, this is the identity for all $s$. For $t=1$, this just gives the circle action, and hence the formulas patch together to give the desired isotopy from $F_{\ba,0}=\id_N$ to $F_{\ba,1}=f_\ba$. 
\end{proof}

\begin{theorem}\label{T:extension} There exists a contractible $W$ and a torus $T\subseteq Y$, where for all $\ba \in H_1(T)$, the diffeomorphism $f_\ba$ extends to a diffeomorphism of $W$, but where $T$ does not extend to a smooth embedding of $S^1 \times D^2$. 
\end{theorem}
\begin{proof}
Note first  that for any $\ba,\ \bb \in H_1(T)$, the corresponding twists satisfy
\begin{equation}\label{E:hom}
f_{\ba + \bb} = f_\ba \circ f_\bb.
\end{equation}
This follows directly from \eqref{E:twist}.  In particular, to prove that every twist along a torus $T$ extends over a fixed manifold $W$, it suffices to show that $f_\ba$ and $f_\bb$ extend, where $\ba$ and $\bb$ form a basis for $H_1(T)$.

We construct a family of examples as in the proof of Theorem~\ref{T:nodehnfortori}. Consider the torus shown in Figure~\ref{F:generaltoruslooptheorem}, with $J$ a non-trivial $(p,q)$ torus knot, $K$ the unknot, and $n= -1$. Let $E$ be the exterior of the knot $J$.  There is a standard circle action on $E$, coming from the circle action on the $3$-sphere given in complex coordinates by  $e^{i \theta} \cdot(z_1,z_2) = (e^{pi\theta}z_1,e^{qi\theta}z_2)$. It is easy to see that $J$ is invariant, and that the orbit of a point on the peripheral torus is given by $\lambda + pq \mu$, where $\lambda$ and $\mu$ are the longitude and meridian of $J$.

Let $N$ be the outer component of $Y -\nhd(T)$; it is given by surgery on $E$ along $K$. In fact (compare~\cite[Chapter 9I]{rolfsen:knots}), $N$ is diffeomorphic to $E$, via a diffeomorphism that is a Dehn twist along the boundary. Under this diffeomorphism, the curve $\alpha$ in Figure~\ref{F:generaltoruslooptheorem} goes to $\mu$, while $\beta$ goes to $\mu + \lambda$. It follows from Lemma~\ref{L:seifert} that the diffeomorphism $f_{\beta+(pq-1)\mu}$ is isotopic to the identity. On the other hand, again using the isotopy of tori from the proof of Theorem~\ref{T:nodehnfortori}, we know from~\cite[Corollary 3.3]{gompf:infinite-cork-handlebody} that $f_\alpha$ extends over $W$.  Since $\alpha$ and $\beta + (pq-1)\mu$ are linearly independent, \eqref{E:hom} implies that $f_\ba$ extends over $W$ for all $\ba \in H_1(T)$.

The fact that $T$ does not bound a solid torus is shown by the calculation in the proof of Theorem~\ref{T:nodehnfortori}. As argued at the end of that proof, if $T$ did bound a solid torus, then both $-J$ and $K \conn J_{-1,1}$ would have to be slice. Since $K$ is the unknot, $K\conn J_{-1,1}$ is simply the reverse of $J$. The result follows since non-trivial torus knots are never algebraically slice. 
\end{proof}

As we mentioned earlier, Gompf's family of infinite corks is of the form given in Figure~\ref{F:gompfcork}, with $J$ in an infinite family of double twist knots and $n\neq 0$, with respect to a function given by the Dehn twist along the longitude of the torus shown in the figure. In~\cite[Question 1.6]{gompf:infinite-cork}, Gompf asks whether \textit{any} knot $J$, with the longitudinal twist, would yield an infinite cork. This question was also raised in~\cite{gompf:infinite-cork-handlebody}. The above theorem shows that the answer is no. Indeed the theorem shows that there exists infinitely many choices for $J$ for which  no Dehn twist along a curve on the torus gives the above $4$-manifold the structure of a cork. Our method of proof does not generalize to knots other than torus knots, since it is known that the complement of a knot is Seifert fibered if and only if it is a torus knot~\cite{moser:seifert-fibered-complement}.

%\bibliography{dehn}
%\bibliographystyle{amsplain}
\providecommand{\bysame}{\leavevmode\hbox to3em{\hrulefill}\thinspace}

\end{document}